\documentclass[12pt]{amsart}

\usepackage{amsmath}
\usepackage{amssymb}
\usepackage{amsfonts}
\usepackage{amsthm}
\usepackage{enumerate}
\usepackage{color}
\usepackage{psfrag}
\usepackage{mathrsfs}
\usepackage{amscd}
\usepackage[hyperindex,backref=page]{hyperref}
\usepackage{graphicx,color}
\usepackage{tikz} 
\usepackage{enumitem}
\usepackage[utf8]{inputenc}

\newtheorem{thm}{Theorem}[section]
\newtheorem{cor}[thm]{Corollary}
\newtheorem{lemma}[thm]{Lemma}
\newtheorem{prop}[thm]{Proposition}

\theoremstyle{definition}

\newtheorem{ex}[thm]{Example}

\newtheorem*{theorem-symbolic Rees}{Theorem \ref{symbolic Rees}}  
\newtheorem*{theorem-limDepth}{Theorem \ref{limDepth}}  
\newtheorem*{maintheorems}{Theorems \ref{degree-six},\ref{degree-four},\ref{degree-three},\ref{degree-two}}

\theoremstyle{remark}

\DeclareMathOperator{\pd}{proj.dim}

\DeclareMathOperator{\depth}{depth}

\DeclareMathOperator{\reg}{reg}
\DeclareMathOperator{\lex}{lex}

\DeclareMathOperator{\set}{set}
\DeclareMathOperator{\Min}{Min}
\def\m{\mathfrak{m}}

\begin{document}


\title[symbolic Rees algebras of complementary edge ideals]{symbolic Rees algebras of complementary edge ideals}

\author{Antonino Ficarra, Somayeh Moradi, and Yuji Muta}

\address[A. Ficarra]{BCAM – Basque Center for Applied Mathematics, Mazarredo 14, 48009 Bilbao, Basque Country – Spain, Ikerbasque, Basque Foundation for Science, Plaza Euskadi 5, 48009 Bilbao, Basque Country – Spain}
\email{aficarra@bcamath.org, antficarra@unime.it}

\address[S. Moradi]{ Department of Mathematics, Faculty of Science, Ilam University, P.O.Box 69315-516, Ilam, Iran}
\email{so.moradi@ilam.ac.ir}

\address[Y. Muta]{Department of Mathematics, Okayama University, 3-1-1 Tsushima-naka, Kita-ku, Okayama 700-8530 Japan.}
\email{p8w80ole@s.okayama-u.ac.jp}

\begin{abstract}
Let $G$ be a finite simple graph on $[n]$ and let $I_c(G)$ denote its complementary edge ideal in the polynomial ring $S = K[x_1,\dots,x_n]$. We give a combinatorial description, in terms of the structure of $G$, of the minimal generators of the symbolic Rees algebra
$\mathcal{R}_s(I_c(G)) = \bigoplus_{k \geq 0} I_c(G)^{(k)} t^k$, and show that this algebra is generated in degree at most $6$. Moreover, we completely determine the minimal generators of $\mathcal{R}_{s}(I_{c}(G))$ in graph-theoretic terms. We then study in more detail the homological invariants of the symbolic powers $I_c(G)^{(k)}$ for the classes of cycle graphs and complete multipartite graphs. For theses families, we study the behavior of the symbolic depth function  $k\mapsto\depth S/I_c(G)^{(k)}$, we obtain the  limit depth of the symbolic powers and the Waldschmidt constant of $I_c(G)$, and further prove that all the symbolic powers $I_c(G)^{(k)}$ are componentwise linear.  
\end{abstract}






\thanks{}
	

\maketitle


\section{Introduction}

Complementary edge ideals constitute a relatively new class of squarefree monomial ideals that naturally extend the interplay between commutative algebra and graph theory.  
Let $G$ be a finite simple graph on the vertex set $V(G)=[n]$ and 
let  $S=K[x_1,\ldots,x_n]$ be the polynomial ring over a field $K$. The {\it complementary edge ideal} of $G$ is defined as $$I_{c}(G)=\left(\frac{x_{1}x_{2}\cdots x_{n}}{x_{i}x_{j}}: \, \{i,j\}\in E(G)\right).$$ This class of monomial ideals was independently introduced in  \cite{fm2025} and \cite{hqm2026}, where a detailed relationship between the algebraic and homological properties of these ideals and the combinatorial structure of their underlying graphs was established. Every squarefree monomial ideal of $S$ generated in degree $n-2$ arises as a complementary edge ideal, and here is what makes this class both natural and until recently,  unexplored: a classical theorem of Herzog, Hibi, and Zheng \cite{HHZ} shows that a squarefree monomial ideal  with a $2$-linear resolution has linear powers. In a classification of the analogous question for general degree $d$, the first and the second author showed in \cite{fmStanley-Reisner2025} that the equivalence ``linear resolution $\iff$ linear powers", uniform across the base field, holds precisely for $d\in\{0,1,2,n-2,n-1,n\}$. Beyond the trivial cases and the classical case $d=2$, the case $d=n-2$ singles out, as the one genuinely new family this classification points to.

Algebraic properties of complementary edge ideals and their ordinary powers have been studied in several papers. In \cite{fm2025} and \cite{hqm2026} combinatorial characterizations have been obtained for important properties such as Cohen--Macaulayness, Gorensteinness, sequentially Cohen--Macaulayness, matroidality, levelness, and linear resolutions of complementary edge ideals. Moreover, in \cite{fm2025} the regularity of ordinary powers of these ideals were described explicitly in terms of graph invariants. Continuing this line of study, the first and the second author studied the structure of the ordinary Rees algebra of complementary edge ideals, its Koszulness and normality criteria in terms of the combinatorial structure of the underlying graph in \cite{fmRees2025}. They also determined the asymptotic behavior of the depth of ordinary powers and a bound for the corresponding depth stability index. Lu, Wang, and Zhu \cite{lwz2025} studied the projective dimension of ordinary powers and the homological shift algebra of complementary edge ideals, and the first author \cite{f2026} recently gave a complete determination of the associated primes of all ordinary powers of this ideal,  showed that they satisfy the persistence property, and further computed the $v$-function of these ideals. Roy and Saha \cite{rs2025} determined the conditions for equality of the ordinary and symbolic powers $I_c(G)^k=I_c(G)^{(k)}$ for all $k$, and Bhardwaj, Chau, and Javadekar~\cite{bcj2026} and Lama~\cite{l2026} characterized licci complementary edge ideals. 

This body of work has so far been concerned mainly with the ordinary powers and the ordinary Rees algebra of $I_c(G)$. Apart from the characterization of the equality between ordinary and symbolic powers obtained in \cite{rs2025}, the symbolic powers of complementary edge ideals and their symbolic Rees algebras remain largely unexplored. 
In this paper, we aim to fill this gap. We describe the generators of the symbolic Rees algebra of $I_c(G)$ for an arbitrary graph $G$ in terms of the graph structure explicitly, and furthermore for the families of cycle graphs and complete multipartite graphs we study homological invariants and properties of symbolic powers of  $I_c(G)$. The regularity and depth of symbolic powers $I_c(G)^{(k)}$ are also studied by the authors in the work~\cite{fmm2026}.    

The symbolic Rees algebra $\mathcal{R}_s(I)=\bigoplus_{k\geq 0}I^{(k)}t^k$ of an ideal $I$ in a Noetherian ring encodes the entire family of symbolic powers and provides a natural algebraic framework for studying the asymptotic and homological behavior of these powers. Nagata's early examples \cite{n1959} already signaled the subtlety of symbolic powers, giving an ideal $I$ in a polynomial ring whose symbolic Rees algebra is not finitely generated. Roberts \cite{R1985} later exhibited a prime ideal in a polynomial ring whose symbolic Rees algebra fails to be Noetherian. 
Against this backdrop, for monomial ideals the situation is far more favorable. In their foundational work, Herzog, Hibi, and N.V. Trung \cite{HHT} proved that the symbolic Rees algebras of monomial ideals are always finitely generated, and if further the ideal is squarefree, the algebra is normal, and Cohen-Macaulay (see also \cite{{l1988}}). An active line of research is then to investigate the generators of the algebra $\mathcal{R}_s(I)$ in terms of the combinatorial objects underlying the given monomial ideal $I$.  
Symbolic Rees algebras of edge ideals and cover ideals of graphs are well studied in the literature, initiated by the work of Simis, Vasconcelos, and Villarreal \cite{SVR},
see for example ~\cite{DV2010,HHT,MRV,MGR2017,V2008}. Notably, Villarreal~\cite{V2008} showed that for any perfect graph $G$, the minimal generators of $\mathcal{R}_{s}(I(G))$ of degree $k$
are in one to one correspondence with the cliques (complete subgraphs) of $G$ of size $k+1$. 

In the case of complementary edge ideals, however, a remarkable phenomenon occurs: independently of the size of the graph $G$, the symbolic Rees algebras  $\mathcal{R}_{s}(I_{c}(G))$ is generated in degree at most 6, and that this bound is sharp.
More precisely, building on Lemmas~\ref{function lemma1} and~\ref{function lemma2}, we show the following theorem as one of the main results of the paper. 

\begin{theorem-symbolic Rees} 
Let $G$ be a graph on the vertex set $[n]$. Then we have 
$$I_{c}(G)^{(k)}=\displaystyle\sum_{t=1}^{6}I_{c}(G)^{(t)}I_{c}(G)^{(k-t)}\mbox{ for all }k\geq7.$$
In particular, the graded $S$-algebra $\mathcal{R}_{s}(I_{c}(G))$ is generated in degree at most 6. More precisely, we have $$\mathcal{R}_{s}(I_{c}(G))=S[I_{c}(G)t,I_{c}(G)^{(2)}t^2, I_{c}(G)^{(3)}t^3,I_{c}(G)^{(4)}t^4, I_{c}(G)^{(6)}t^6].$$
\end{theorem-symbolic Rees}

The proof of Theorem~\ref{symbolic Rees} plays a crucial role in determining an explicit description of the generators of $\mathcal{R}_{s}(I_{c}(G))$. In  
Theorems~\ref{degree-six},~\ref{degree-four},~\ref{degree-three}, and~\ref{degree-two} we then exhibit the precise description of the minimal generators of $\mathcal{R}_{s}(I_{c}(G))$ of different degrees, in terms of graph-theoretic information. 

\begin{maintheorems}
	Let $G$ be a graph on $[n]$. Then
    \begin{enumerate}
        \item[(a)] $ut^6$ is a minimal generator of  $\mathcal{R}_{s}(I_{c}(G))$   if and only if $u={\bf x}_A^2{\bf x}_B^3{\bf x}_C^4$, where  $A$ is a clique of $G$ with $|A|\geq 3$, $B=\{j\in[n]\setminus A:\, \{i,j\}\in E(G) \text{ for any } i\in A\}\neq \emptyset,$ $(G_{B})^c$ is a non-bipartite graph without isolated vertices,  $C=[n]\setminus (A\cup B)$, and one of the following conditions holds: 
	\begin{enumerate}
		\item [\textup{(i)}] $|A|\geq 4$.
		\item [\textup{(ii)}]   $|A|=3$ and  $G_{A}$ is not a split-dominating triangle of $G$.
	\end{enumerate}  
        
        \item[(b)] $ut^4$ is a minimal generator of  $\mathcal{R}_{s}(I_{c}(G))$ if and only if $u=x_rx_s{\bf x}_B^2{\bf x}_C^3$, where $\{r,s\}\in E(G)$ is not a dominating edge of $G$, $B=N_G(r)\cap N_G(s)\neq \emptyset$, $B$ is not a clique of $G$ and $C=[n]\setminus (\{r,s\}\cup B)$.
        
        \item[(c)] $ut^3$ is a minimal generator of  $\mathcal{R}_{s}(I_{c}(G))$ if and only if $u={\bf x}_A{\bf x}_B^2$, where $A$ is a maximal clique of $G$  and $B=[n]\setminus A$ such that either $|A|\geq 4$, or $|A|=3$ and $A$ contains no cone vertex. 

\item[(d)] $ut^2$ is a minimal generator of  $\mathcal{R}_{s}(I_{c}(G))$ if and only if $u$ is of one of the following  forms:
\begin{enumerate}
    \item[\textup{(i)}]  $u={\bf x}_A{\bf x}_B^2$, where $A=N_G(i)$ for some vertex $i$ with $\deg_G(i)\geq 3$ and $B=[n]\setminus N_G[i]$.  
    \item[\textup{(ii)}] $u={\bf x}_{[n]}$, if $G$ has no cone, and if the matching number of $G$ is one, when $n=4$.   
\end{enumerate}   
    \end{enumerate}
\end{maintheorems}

Consequently, Corollaries~\ref{no-degree-six},~\ref{no-degree-four}, and~\ref{no-degree-three} provide explicit criteria, expressed in terms of the existence of certain subgraph structures, for the absence of minimal generators in degrees 
$6$, $4$ and $3$, respectively.
Moreover, Corollary~\ref{degree-at-most-two} determines combinatorial conditions under which $\mathcal{R}_{s}(I_{c}(G))$ is generated in degree at most $2$.
Finally, in Corollary~\ref{equlaityPowes}, we recover~\cite[Theorem 2.9]{rs2025}, which characterizes graphs $G$ for which $I_c(G)^{(k)}=I_c(G)^k$ for all $k\geq 1$, which is the case when  $\mathcal{R}_{s}(I_{c}(G))$ is standard graded. 

\smallskip

In Section~\ref{cycles}, we concentrate on cycle graphs, and in Theorem~\ref{main-cycle} we determine several algebraic invariants of symbolic powers of $I=I_c(G)$ when $G=C_n$ is a cycle on $n$ vertices. Namely, we give explicit descriptions for the graded Betti numbers of the symbolic powers $I^{(k)}$.
Moreover, we show that $\depth S/I^{(k)}=\lim_{k\rightarrow\infty}\depth S/I^{(k)}=n-3$, for $n\geq 4$ and all $k\ge1$. 
For any $n\geq5$,  the Hilbert series of the symbolic fiber $\mathcal{F}_s(I)=\mathcal{R}_s(I)/\mathfrak{m}\mathcal{R}_s(I)$  is shown to be of the form $$
        F(\mathcal{F}_{s}(I),z)=\dfrac{1+(n-2)z+z^2}{(1-z)^{3}(1+z)}.
        $$
Furthermore, the Waldschmidt constant of $I$ is obtained, and it is shown that $I^{(k)}$ has linear quotients for all $k\geq 1$. 

Section~\ref{CompleteBipartiteG} is devoted to the study of the symbolic powers $I_c(G)^{(k)}$, when $G$ is a complete multipartite graph. The main objective of this section is the behavior of the depth function $k\mapsto\depth S/I_c(G)^{(k)}$, the Waldschmit constant and the componentwise linearity of symbolic powers of $I_c(G)$. 
Firstly, 
Theorem~\ref{multipartitre-Rees} presents the structure of $\mathcal{R}_{s}(I_{c}(G))$ for this family of graphs. When $I$ is a squarefree monomial ideal, then it follows from \cite[Theorem 4.7]{HT2010}  that  $\depth S/I^{(k)}$ stabilizes for large enough $k$. Furthermore, it was shown in \cite[Theorem 2.4]{hktt2017} that  this stable value is equal to $\dim(S)-\ell_s(I)$, where $\ell_s(I)$ denotes the symbolic analytic spread of $I$, that is the Krull dimension of the fiber $\mathcal{F}_s(I)$. By interpreting $\ell_s(I)$ as the Krull dimension of an appropriate Veronese subalgebra of $\mathcal{F}_s(I)$, and studying the Hiblert polynomial of this subalgebra for $I=I_c(G)$, we obtain

\begin{theorem-limDepth}
Let $G=K_{d_1,\dots,d_m}$ be a complete multipartite graph on $n\geq 3$ vertices. Let $s=|\{i\,\,: d_{i}=1\}|$ and $t=|\{i\,\,: d_{i}\geq2\}|$. Then 
$$\lim_{k\rightarrow\infty}\depth S/I_{c}(G)^{(k)}=\begin{cases}
n-m-1, & \text{if } d_i\geq 2 \text{ for all } i, \\
n-t-\min\{3,s\}, & \text{otherwise}. 
\end{cases}$$ 
Moreover, this value is a lower bound for  $\depth S/I_{c}(G)^{(k)}$ for all $k\geq 1$.
\end{theorem-limDepth}

In Theorem~\ref{LQ of complete multi} we show that all the symbolic powers $I_{c}(G)^{(k)}$ have linear quotients, where $G=K_{d_1,\dots,d_m}$. This result gives enough tools to determine the sequences of depth of symbolic powers of $I_{c}(G)$ under the assumption that $d_{i}\geq2$ for all $i$ (see Proposition~\ref{depth of complete multi}) as follows.
\begin{equation}\label{depthSequence}
\depth S/I_{c}(G)^{(k)}=
\begin{cases}
n-3, & \text{if } k=1,2,4, \\
n-m-1, & \text{otherwise}. 
\end{cases}   
\end{equation}

It is natural to ask whether the symbolic depth function is non-increasing. For cover ideals of graphs, this is known to be the case \cite[Theorem 3.2]{hktt2017}. On the other hand, the corresponding statement fails for squarefree monomial ideals in general \cite{nt2019,mt2025}, while it remains an open question for symbolic powers of edge ideals. Our formula \eqref{depthSequence} shows that monotonicity also fails for complementary edge ideals. Indeed, if $m\geq 3$, then we have $$\depth S/I_{c}(G)^{(3)}=n-m-1<n-3=\depth S/I_{c}(G)^{(4)}.$$

Finally, Proposition \ref{Waldschmit-constant} determines the Waldschmidt constant of $I_c(G)$ in terms of the number of vertices of $G$. The proof relies crucially on the characterization of the minimal monomial generators of $I_{c}(G)^{(k)}$.


\section{Preliminaries}
In this section, we review  definitions and known results from combinatorial commutative algebra that will be used throughout the paper. We refer the reader to \cite{BHBook, HHBook, VBook} for the detailed information on combinatorial commutative algebra. 

Throughout this paper, $G = (V(G), E(G))$ denotes a finite simple graph on vertex set $V(G) = [n] = \{1,\dots,n\}$, without isolated vertices, and $S = K[x_1,\dots,x_n]$ denotes the polynomial ring over a field $K$.  

We write $G^c$ for the complement of $G$, i.e.\ the graph on $[n]$ with $\{i,j\} \in E(G^c)$ if and only if $\{i,j\} \notin E(G)$.
For any $i\in [n]$, we set $N_G(i)=\{j\in [n]:\; \{i,j\}\in E(G)\}$ and $N_G[i]=N_G(i)\cup\{i\}$. The induced subgraph of $G$ on a subset $A\subset [n]$ is denoted by $G_A$.

For $F \subseteq [n]$ we write ${\textbf x}_F = \prod_{i \in F} x_i$. 	
The \emph{complementary edge ideal} of $G$ is the squarefree monomial ideal
 	\[
 	I_c(G) = \big({\textbf x}_{[n]\setminus\{i,j\}} : \{i,j\} \in E(G)\big) \subseteq S.
 	\]

Let $\mathcal{T}(G)$ denotes the set of all triangles ($K_3$-subgraphs) of $G$. The following result determines the minimal primary decomposition of a complementary edge ideal. 

\begin{thm}[\cite{fm2025}, Theorem 2.1, \cite{hqm2026}, Theorem 1.1]\label{min prime}
Let $G$ be a finite simple graph without isolated vertices. Then 
$$I_{c}(G)=(\bigcap_{e\in E(G^{c})}P_{e})\cap(\bigcap_{T\in\mathcal{T}(G)}P_{T}),$$
where $P_A=(x_i:\, i\in A)$ for any $A\subseteq [n]$. 
\end{thm}

Let $I \subseteq S$ be an ideal, and let $\operatorname{Ass}(I)$ denote the set of its associated primes. For $k\in \mathbb{N}$, the $k$-th \emph{symbolic power} of $I$ is defined to be
 	\[
 	I^{(k)} = \bigcap_{P \in \operatorname{Ass}(I)} \big(I^k S_P \cap S\big).
 	\]
We note that an alternative definition of symbolic powers exists, obtained by replacing the set of associated primes $\operatorname{Ass}(I)$ with the set of minimal primes $\operatorname{Min}(I)$ in the above intersection. When 
$I$ has no embedded primes, as is the case for squarefree monomial ideals, the two definitions coincide.  For a squarefree monomial ideal $I$ with the set of minimal prime ideals $\operatorname{Min}(I) = \{P_1,\dots,P_r\}$, the symbolic powers admit the explicit description $I^{(k)} = P_1^{k} \cap \cdots \cap P_r^{k}$.
This together with Theorem~\ref{min prime} yields

\begin{equation}\label{sym-comp}
I_{c}(G)^{(k)}=(\bigcap_{e\in E(G^{c})}P_{e}^{k})\cap(\bigcap_{T\in\mathcal{T}(G)}P_{T}^{k}).
\end{equation}

The \emph{symbolic Rees algebra} of an ideal $I \subseteq S$ is the graded $S$-subalgebra of $S[t]$
 	\[
 	\mathcal{R}_s(I) = \bigoplus_{k \geq 0} I^{(k)} t^k \subseteq S[t],
 	\]
 with the convention $I^{(0)} = S$. By a theorem of Lyubeznik \cite{l1988}, if $I$ is a squarefree monomial ideal then $\mathcal{R}_s(I)$ is a finitely generated $K$-algebra, generated by a unique minimal set of monomials of $S[t]$. This fact was later  extended to all monomial ideals by Herzog, Hibi, and N.V. Trung \cite{HHT} via Gordan's lemma.
It follows that for a monomial ideal $I$,   $\mathcal{F}_s(I)$ is finite dimensional. The {\em symbolic analytic spread} of $I$ is then defined to be $\ell_s(I)=\dim \mathcal{F}_s(I)$.

When $I$ is a squarefree monomial ideal, it is known by \cite[Theorem 4.7]{HT2010}  that  $\depth S/I^{(k)}$ stabilizes for large enough $k$. The stable value of this depth sequence is called the {\em limit depth} of symbolic powers of $I$. The following result by Hoa, Kimura, Terai and T.N. Trung \cite[Theorem 2.4]{hktt2017} will be used in the paper.
\begin{thm}\label{limit}
 Let $I\subset S$ be a squarefree monomial ideal. Then
\begin{enumerate}
 \item[\textup{(1)}] $\depth S/I^{(k)}\geq \dim(S)-\ell_s(I)$ for all $k\geq 1$. 
\item [\textup{(2)}]$\lim_{k\rightarrow\infty}\depth S/I^{(k)}=\dim(S)-\ell_s(I)$.
\end{enumerate}
\end{thm}

The set of minimal monomial generators of a monomial ideal $I$ is denoted by $\mathcal{G}(I)$.
A monomial ideal $I$ with   $\mathcal{G}(I)=\{u_1,\dots,u_r\}$ is said to have \emph{linear quotients} if there is an ordering $u_1,\dots,u_r$ of the generators such that, for every $2 \leq k \leq r$, the colon ideal $(u_1,\dots,u_{k-1}) : u_k$
 is generated by a subset of the variables $x_1,\dots,x_n$.  If $I$ has linear quotients with respect to an ordering $u_1,\dots,u_r$ of its minimal generators, then $I$ is componentwise linear, and its graded Betti numbers are determined explicitly by the sets $\set(u_k) = \{i : x_i \in (u_1,\dots,u_{k-1}):u_k\}$. In particular the projective dimension and Castelnuovo--Mumford regularity of $S/I$ can be read off from this combinatorial data, see~\cite[Corollary 8.2.2]{HHBook}.

For a homogeneous ideal $I \subseteq S$, the \emph{initial degree} of $I$ is $\alpha(I) = \min\{d : I_d \neq 0\}$.
 	The \emph{Waldschmidt constant} of $I$ 
    is defined as
 	\[
 	\widehat{\alpha}(I) = \lim_{k \to \infty} \frac{\alpha(I^{(k)})}{k}.
 	\]


\section{Symbolic Rees algebra of $I_{c}(G)$}\label{sym-Rees}

In this section we investigate the symbolic Rees algebra $\mathcal{R}_{s}(I_{c}(G))$. We show that it is generated in degree at most $6$ and establish its minimal generators in combinatorial terms. We further give criteria for $\mathcal{R}_{s}(I_{c}(G))$ to be generated in smaller degrees. 

The following two lemmas are essential to achieve our goal.

\begin{lemma}\label{function lemma1}
Let $k\geq7$ and $t\in[6]$. Let $f:\mathbb{N}\rightarrow\mathbb{N}$ be a map satisfying the following conditions:
\begin{enumerate}
\item[\textup{(1)}] $f(0)=0$, 
\item[\textup{(2)}] $0\leq f(s+1)-f(s)\leq1$ for all $s\in\mathbb{N}$, 
\item[\textup{(3)}] $f(s)=t$ for all $s\geq k$, 
\item[\textup{(4)}] $f(s)+f(k-s)=t$ for all $0\leq s\leq k$. 
\end{enumerate}
Then the following statements hold: 
\begin{enumerate}[label=(\alph*)]
\item[\textup{(a)}] If $x,y\in\mathbb{N}$ with $x\geq y$, then we have $f(x)-f(y)\leq x-y$. 
\item[\textup{(b)}] If $x+y\geq k$ for some $x,y\in\mathbb{N}$, then we have $t\leq f(x)+f(y)\leq t+x+y-k$. 
\item[\textup{(c)}] Suppose that $\alpha+\beta+\gamma=\ell$ and $f(\alpha)+f(\beta)+f(\gamma)=t$ for some $\alpha,\beta,\gamma\in\mathbb{N}$. If $x\geq\alpha, y\geq\beta$, and $z\geq\gamma$ for some $x,y,z\in\mathbb{N}$, then we have $t\leq f(x)+f(y)+f(z)\leq t+x+y+z-\ell$. 
\end{enumerate}
\end{lemma}
\begin{proof}
(a) Fix $x,y\in\mathbb{N}$ with $x\geq y$. Then, by condition (2) we have $$f(x)-f(y)=\sum_{r=y}^{x-1}(f(r+1)-f(r))\leq\sum_{r=y}^{x-1}1=x-y.$$

Note that, by condition (1), (a), and the assumption that $k\geq7$, it follows that $f(s)\leq s$ for any $s\in\mathbb{N}$. 

(b) Suppose that $x+y\geq k$. If $x\geq k$, then, by condition (3), we have $f(x)=t$, which implies $t\leq f(x)+f(y)$. Suppose that $x<k$. Then, since $x+y\geq k$, we have $y\geq k-x$. By (a), we have $f(y)\geq f(k-x)$, and hence, we obtain $t= f(x)+f(k-x)\leq f(x)+f(y)$ by condition (4). Thus, we get $t\leq f(x)+f(y)$. Next, we prove that $ f(x)+f(y)\leq t+x+y-k$. If $x\geq k$, then, by condition (3), we have $f(x)=t$, and thus, we obtain $f(x)+f(y)\leq t+y\leq t+x+y-k$. Hence, we may assume that $x<k$ and $y<k$. In this case, since $x+y\geq k$, we have $k-x\leq y$, which implies $f(y)-f(k-x)\leq y-(k-x)$ by (a). By condition (4), we obtain $f(x)+f(y)\leq f(x)+f(k-x)+y-(k-x)=t+x+y-k$, as required. 

(c) Fix $x\geq\alpha$, $y\geq\beta$, and $z\geq\gamma$. Then, by the assumption, we get $f(x)+f(y)+f(z)\geq f(\alpha)+f(\beta)+f(\gamma)=t$. Now, by (a), we have $f(x)-f(\alpha)\leq x-\alpha$, $f(y)-f(\beta)\leq y-\beta$, and $f(z)-f(\gamma)\leq z-\gamma$. Hence, we obtain $f(x)+f(y)+f(z)-\{f(\alpha)+f(\beta)+f(\gamma)\}\leq x+y+z-(\alpha+\beta+\gamma)$, which implies the desired inequality. 
\end{proof}

\begin{lemma}\label{function lemma2}
Let $k\geq7$ and let $S\subset[k]$ with $|S|\leq6$. Define the map $f_{S}:\mathbb{N}\rightarrow\mathbb{N}$ by
\[
f_{S}(s)=
\begin{cases}
|S\cap [s]|, & \text{if } 0\leq s<k,\\
|S|, & \text{if } s\geq k.
\end{cases}
\]
Assume that $s\in S$ if and only if $k+1-s\in S$. Then $f_{S}$ satisfies the conditions of Lemma \ref{function lemma1}, setting $t=|S|$. 
\end{lemma}

\begin{proof}
We verify  conditions (1) to (4) of Lemma \ref{function lemma1}. Set $f=f_{S}$. Clearly, (1) and (3) hold. Hence, it is enough to verify that (2) and (4) hold. 

(2) Fix $s\in\mathbb{N}$. If $s< k$, then we have $f(s)=|S\cap[s]|\leq |S\cap[s+1]|\leq|S\cap[s]|+1\leq f(s)+1$, which implies that $0\leq f(s+1)-f(s)\leq1$. Also, if $s\geq k$, then we have $f(s)=|S|=f(s+1)$, and hence, the desired inequality hold. Suppose that $s=k-1$. Then, by the definition, we have $f(s)=|S\cap[k-1]|$ and $f(s+1)=f(k)=|S|$. Hence, we obtain $f(s+1)-f(s)=|S|-|S\cap[k-1]|=|S\cap\{k\}|\in\{0,1\}$, as required. 

(4) Fix  $0\leq s\leq k$. If $s=0,k$, then we have $f(s)+f(k-s)=|S|$. Suppose that $0<s<k$. Then we have $f(s)=|S\cap[s]|$ and $f(k-s)=|S\cap[k-s]|$. By the assumption that $s\in S$ if and only if $k+1-s\in S$, we have a bijection $S\cap[s]\rightarrow S\cap[k+1-s,k+2-s,\ldots,k]$, which maps $r$ to $k+1-r$ for each $r\in S\cap[s]$. Therefore, we see that $f(s)+f(k-s)=|S\cap[s]|+|S\cap[k-s]|=|S\cap[k]|=|S|$, as required. 
\end{proof}

\begin{thm}\label{symbolic Rees}
Let $G$ be a graph on the vertex set $[n]$. Then we have 
$$I_{c}(G)^{(k)}=\displaystyle\sum_{t=1}^{6}I_{c}(G)^{(t)}I_{c}(G)^{(k-t)}\mbox{ for all }k\geq7.$$
In particular, the graded $S$-algebra $\mathcal{R}_{s}(I_{c}(G))$ is generated in degree at most 6. More precisely, we have $$\mathcal{R}_{s}(I_{c}(G))=S[I_{c}(G)t,I_{c}(G)^{(2)}t^2, I_{c}(G)^{(3)}t^3,I_{c}(G)^{(4)}t^4, I_{c}(G)^{(6)}t^6].$$
\end{thm}
\begin{proof}
Fix $k\geq7$ and ${\bf x}^{\bf a}\in I_{c}(G)^{(k)}$. From (\ref{sym-comp}), we see that ${\bf x}^{\bf a}=x_{1}^{a_{1}}x_{2}^{a_{2}}\cdots x_{n}^{a_{n}}\in I_{c}(G)^{(k)}$ if and only if $a_{i}+a_{j}\geq k$ for any $\{i,j\}\notin E(G)$ and $a_{i}+a_{j}+a_{\ell}\geq k$ for any pairwise distinct vertices $i,j$ and $\ell$. It is enough to prove that there exist $t\in[6]$ and ${\bf b}=(b_{1},\ldots, b_{n})\in\mathbb{N}^{n}$ such that $0\leq b_{i}\leq a_{i}$ for all $i$, $t\leq b_{i}+b_{j}\leq t+a_{i}+a_{j}-k$ for each $\{i,j\}\notin E(G)$, and $t\leq b_{i}+b_{j}+b_{\ell}\leq t+a_{i}+a_{j}+a_{\ell}-k$ for each pairwise distinct vertices $\{i,j,\ell\}$. Indeed, if this is true, then we have ${\bf x}^{\bf b}\in I_{c}(G)^{(t)}$ and ${\bf x}^{{\bf a}-{\bf b}}\in I_{c}(G)^{(k-t)}$, which implies ${\bf x}^{\bf a}={\bf x}^{\bf b}{\bf x}^{{\bf a}-{\bf b}}\in I_{c}(G)^{(t)}I_{c}(G)^{(k-t)}$. Take the three smallest values among $a_{1},\ldots, a_{n}$, counted with multiplicity, and denote them by $\alpha\leq\beta\leq\gamma$. Then we have $\alpha+\beta+\gamma\geq k$. We distinguish the following two cases: 

\textbf{Case 1.} Suppose $\alpha+\beta+\gamma>k$. In this case, define the map $f:\mathbb{N}\rightarrow\{0,1,2\}$ as 
\[
f(s)=
\begin{cases}
0, & \text{if } s=0,\\
1, & \text{if } 1\leq s<k,\\
2, & \text{if } s\geq k.
\end{cases}
\]
Set $b_{i}=f(a_{i})$ and ${\bf b}=(b_{1},\ldots, b_{n})$. Then we can easily verify that $f$ satisfies the conditions of Lemma \ref{function lemma1}, which implies that $0\leq b_{i}\leq a_{i}$ for all $i$, $2\leq b_{i}+b_{j}\leq 2+a_{i}+a_{j}-k$ for each $\{i,j\}\notin E(G)$. 
Fix pairwise distinct vertices $i,j,\ell$. We may assume that $a_{i}\leq a_{j}\leq a_{\ell}$. Note that $k+1\leq\alpha+\beta+\gamma\leq a_{i}+a_{j}+a_{\ell}$. Suppose on the contrary that $b_{i}+b_{j}+b_{\ell}\leq1$. Then we may assume that $b_{i}\leq 1$, $b_{j}=0$, and $b_{\ell}=0$. These together with the definition of $f$ imply $a_{i}\leq k-1$, $a_{j}=0$, and $a_{\ell}=0$. Hence, we obtain $k+1\leq a_{i}+a_{j}+a_{\ell}\leq k-1$, which is a contradiction. Therefore, we get $2\leq b_{i}+b_{j}+b_{\ell}$. Set $p=|\{x\,: x\in\{a_{i},a_{j},a_{k}\}\mbox{ with }x\geq k\}|$ and $q=|\{x\,: x\in\{a_{i},a_{j},a_{\ell}\}\mbox{ with }1\leq x\leq k-1\}|$. Then we can check that $f(a_{i})+f(a_{j})+f(a_{\ell})=2p+q$ holds. If $p=0$, then we have $f(a_{i})+f(a_{j})+f(a_{\ell})=q\leq 3\leq2+a_{i}+a_{j}+a_{\ell}-k$. Also, if $p=1$, then $a_{i}+a_{j}+a_{\ell}\geq k+q$, and so we have $f(a_{i})+f(a_{j})+f(a_{\ell})=2+q\leq2+a_{i}+a_{j}+a_{\ell}-k$. Finally, suppose $2\leq p\leq 3$. In this case, we clearly have $a_{i}+a_{j}+a_{\ell}\geq2k$, which gives us $2+a_{i}+a_{j}+a_{\ell}-k\geq2+k\geq9$, by the assumption $k\geq7$. Hence, we obtain $f(a_{i})+f(a_{j})+f(a_{\ell})\leq 6<9\leq2+a_{i}+a_{j}+a_{\ell}-k$. 
Therefore, $2\leq b_{i}+b_{j}+b_{k}=f(a_{i})+f(a_{j})+f(a_{\ell})\leq 2+a_{i}+a_{j}+a_{\ell}-k$, and thus we get  ${\bf x}^{\bf a}={\bf x}^{\bf b}{\bf x}^{{\bf a}-{\bf b}}\in I^{(2)}I^{(k-2)}$, as desired. 

\textbf{Case 2.} Suppose $\alpha+\beta+\gamma=k$. In view of Lemma~\ref{function lemma2}, it suffices to prove that there exists a set $S\subset[k]$ satisfying the following conditions:
\begin{enumerate}
\item $|S|\leq 6$,
\item $f_{S}(\alpha)+f_{S}(\beta)+f_{S}(\gamma)=|S|$,
\item $s\in S$ if and only if $k+1-s\in S$.
\end{enumerate}
Indeed, assume that such a subset $S\subset[k]$ exists. Then by setting $b_{i}=f_{S}(a_{i})$ for every $i$, and using Lemma \ref{function lemma1} and \ref{function lemma2}, we obtain that $0\leq b_{i}\leq a_{i}$ for all $i$, $|S|\leq b_{i}+b_{j}\leq |S|+a_{i}+a_{j}-k$ for each $\{i,j\}\notin E(G)$, and $|S|\leq b_{i}+b_{j}+b_{\ell}\leq |S|+a_{i}+a_{j}+a_{\ell}-k$ for each pairwise distinct vertices $\{i,j,\ell\}$. Hence, ${\bf x}^{\bf a}={\bf x}^{\bf b}{\bf x}^{{\bf a}-{\bf b}}\in I^{(|S|)}I^{(k-|S|)}$, with $|S|\leq6$, as desired. To show the existence of such a set $S$ we distinguish the following cases. 

(i) Suppose that $\alpha<\beta$. In this case, we set $S=\{\alpha+1, k-\alpha\}\subset[k]$.

(ii) Suppose that $\alpha=\beta<\gamma$ and $\gamma\geq2\alpha+1$. In this case, we set $S=\{2\alpha+1, \gamma\}\subset[k]$.  

(iii) Suppose that $\alpha=\beta<\gamma$ and $\gamma\leq2\alpha$. In this case, we set $S=\{\alpha,\alpha+1,\alpha+\gamma,\alpha+\gamma+1\}\subset[k]$. 

(iv) Suppose that $\alpha=\beta=\gamma$ and $\alpha$ is odd. In this case, we set $S=\{1,\frac{3\alpha+1}{2},3\alpha\}\subset[k]$.  

(vi) Suppose that $\alpha=\beta=\gamma$ and $\alpha$ is even. In this case, we set $S=\{1,2,\alpha+1,2\alpha,3\alpha-1,3\alpha\}\subset[k]$.  

Then one can easily verify that the set $S$  satisfies the above conditions (1) to (3), with $|S|\in\{2,3,4,6\}$. Hence, the assertion follows. 
\end{proof}

The following example shows that the degree bound for generators of $\mathcal{R}_{s}(I_{c}(G))$ established in Theorem \ref{symbolic Rees} is indeed sharp. This gives an instance for the results developed in the remainder of this section, giving precise descriptions of minimal generators of $\mathcal{R}_{s}(I_{c}(G))$ in terms of $G$. 

\begin{ex}
Let $G$ be the graph with the vertex set $[7]$ and the edge set $E(G)=\binom{[7]}{2}\setminus\{\{5,6\},\{5,7\},\{6,7\}\}$. Notice that the induced subgraph of $G$ on $[4]$ is a complete graph, and each of the vertices $5,6$ and $7$ are adjacent to all vertices in $[4]$ (compare to Theorem~\ref{degree-six}). By a straightforward calculation,  we have $$(x_{1}x_{2}x_{3}x_{4})^2. (x_{5}x_{6}x_{7})^{3}\in I_{c}(G)^{(6)}\setminus\displaystyle\sum_{t=1}^{5}I_{c}(G)^{(t)}I_{c}(G)^{(6-t)}.$$ 
\end{ex}

A triangle $T$ in a graph $G$ is called a {\em split-dominating triangle}, if there exists $i_1\in T$ such that $V(G)=N_G(i_1)\cup[N_G(i_2)\cap N_G(i_3)]$, where $\{i_2,i_3\}=V(T)\setminus \{i_1\}$. A \emph{clique} of $G$ is a subset $Q \subseteq V(G)$ such that $G_{Q}$ is a complete graph.  

\begin{thm}\label{degree-six} 
		Let $G$ be a graph on $[n]$. Then $ut^6$ is a minimal generator of  $\mathcal{R}_{s}(I_{c}(G))$   if and only if $u={\bf x}_A^2{\bf x}_B^3{\bf x}_C^4$, where  $A$ is a clique of $G$ with $|A|\geq 3$, $$B=\{j\in[n]\setminus A:\, \{i,j\}\in E(G) \text{ for any } i\in A\}\neq \emptyset,$$ $(G_{B})^c$ is a non-bipartite graph without isolated vertices,  $C=[n]\setminus (A\cup B)$, and one of the following conditions holds: 
	\begin{enumerate}
		\item [\textup{(i)}] $|A|\geq 4$.
		\item [\textup{(ii)}]   $|A|=3$ and  $G_{A}$ is not a split-dominating triangle of $G$.
	\end{enumerate}  
\end{thm}

\begin{proof}
Set $V(G)=[n]$ and $I=I_c(G)$. First assume that 
$ut^6$ is a minimal generator of  $\mathcal{R}_{s}(I_{c}(G))$.  Let $u={\textbf x}^{\textbf a}$, and let $\alpha,\beta$ and $\gamma$ be as in the proof of Theorem \ref{symbolic Rees}. 	
Considering the proof of Theorem~\ref{symbolic Rees},   
the minimality of $ut^6$ implies that condition (vi) of the proof happens, that is $\alpha+\beta+\gamma=k=6$,  $\alpha=\beta=\gamma$ and $\alpha$ is even. Then we have $\alpha=\beta=\gamma=2$. Let $A=\{i:\, a_i=2\}$, $B=\{i:\, a_i=3\}$, and $C=[n]\setminus (A\cup B)$. Since $\alpha=\beta=\gamma=2$, we have $m=|A|\geq 3$.
Moreover, from $u\in I^{(6)}$ we obtain that $A$ is a clique, and $\{i,j\}\in E(G)$ for any $i\in A$ and any $j\in B$. If $j\in [n]\setminus A$ be such that $\{i,j\}\in E(G)$ for any $i\in A$, then $a_j=3$, otherwise if $a_j\geq 4$, then $u/x_j\in I^{(6)}$, a contradiction. Thus $B=\{j\in[n]\setminus A:\, \{i,j\}\in E(G) \text{ for any } i\in A\}$.  
Also the minimality of $ut^6$ yields that $a_i=4$ for any $i\in C$. Thus $u={\textbf x}_A^2{\textbf x}_B^3{\textbf x}_C^4$.
Now, by contradiction, assume that 
$B=\emptyset$. Then $ut^6=(({\textbf x}_A {\textbf x}_C^2)t^3)^2\in (I^{(3)}t^3)^2,$
which contradicts to the minimality of $ut^6$.
Next, assume on the contrary that $(G_{B})^c$ is a bipartite graph. Then $B$ has a partition $B=B_1\sqcup B_2$ such that $G_{B_1}$ and $G_{B_2}$ are complete graphs. Then   
$$ut^6=(\prod_{i\in A} x_i \prod_{i\in B_1} x_i^2 \prod_{i\in B_2} x_i  \prod_{i\in C} x_i^2)t^3(\prod_{i\in A} x_i \prod_{i\in B_1} x_i \prod_{i\in B_2}x_i^2\prod_{i\in C} x_i^2)t^3\in (I^{(3)}t^3)^2,$$
a contradiction. 
Thus $(G_{B})^c$ is not a bipartite graph, as claimed. Suppose by contradiction that $(G_{B})^c$ has an isolated vertex, say $i$. Then $\{i,j\}\in E(G)$ for any $j\in A\cup B$. This implies that $u/x_i\in I^{(6)}$, which means $ut^6\in \m I^{(6)}t^6$, a contradiction. Thus $(G_{B})^c$ has no isolated vertices.  If $m\geq 4$, then condition (i) holds and we are done. So we may assume that $m=3$. We need to show that $G_{A}$ is not  a split-dominating triangle of $G$. Suppose on the contrary that $G_{A}$ is split-dominating. So we may assume that $A=[3]$, and $V(G)=N_G(1)\cup[N_G(2)\cap N_G(3)]$.
Set $$w=(x_2x_3 \prod_{i\in N_G(1)\setminus A} x_i\prod_{i\notin N_G(1)} x_i^2)t^2(x_1^2x_2x_3\prod_{i\in B} x_i^2 \prod_{i\in N_G(1)\setminus (A\cup B)} x_i^3  \prod_{i\notin N_G(1)} x_i^2)t^4.$$
	Since for any $i\notin N_G(1)$ we have $i\in N_G(2)\cap N_G(3)$, it is seen that $w\in I^{(2)}t^2.I^{(4)}t^4$ and $w$ divides $ut^6$, a contradiction. Hence, condition (ii) holds.
	
	\smallskip

Conversely, consider a monomial $u={\textbf x}_A^2{\textbf x}_B^3{\textbf x}_C^4$,  where  $A$ is a clique of $G$ with $m=|A|\geq 3$, $B=\{j\in[n]\setminus A:\, \{i,j\}\in E(G) \text{ for any } i\in A\}\neq\emptyset$, such that $(G_{B})^c$ is not a bipartite graph,  and $C=[n]\setminus (A\cup B)$.
It is easily seen that $u\in I^{(6)}$. In light of Theorem~\ref{symbolic Rees}, in order to show that $ut^6$ is a minimal generator of  $\mathcal{R}_{s}(I_{c}(G))$,  it suffices to show that $u\notin I^{(2)}I^{(4)}+(I^{(3)})^2+\m I^{(6)}$. First we show that $u\notin \m I^{(6)}$. For any $i\in A$, we have $a_i+a_j+a_r=6$ for some distinct $j,r\neq i$. So $u/x_i\notin I^{(6)}$. 
Since $(G_{B})^c$ has no isolated vertices, for any $i\in B$, there exists $j\in B$ with $j\neq i$ such that $\{i,j\}\notin E(G)$. Since $a_i+a_j=6$, we conclude that $u/x_i\notin I^{(6)}$. Now, let $i\in C$. Then there exists $j\in A$ such that $\{i,j\}\notin E(G)$. This together with the equality $a_i+a_j=6$ gives $u/x_i\notin I^{(6)}$. Altogether, we conclude that $u\notin \m I^{(6)}$.  
Next we show that $u\notin I^{(2)}I^{(4)}+(I^{(3)})^2$.

First assume that (i) holds. 
Suppose to the contrary that $u\in I^{(2)}I^{(4)}$.  Then $u=vw$ for some monomials $v\in I^{(2)}$ and $w\in I^{(4)}$. Let $v={\bf x}^{\bf b}$ and $w={\bf x}^{\bf c}$.
By (i) we have $m\geq 4$. Without loss of generality assume that $[4]\subseteq [m]$.
Then $b_1+b_2+b_3\geq 2$ and $c_1+c_2+c_3\geq 4$. Note that $b_i+c_i=2$ for any $1\leq i\leq 4$. Thus    $c_1+c_2+c_3=6-(b_1+b_2+b_3)\leq 4$. So  $c_1+c_2+c_3=4$ and $b_1+b_2+b_3=2$. Hence, we have $b_i=0$ for some $i\in[3]$. We may assume that $b_1=0$ and so $b_2+b_3=2$. Similar argument shows that $b_2+b_3+b_4=2$, which implies that $b_4=0$. So $2\leq b_2+b_1+b_4=b_2\leq 2$ and hence, $b_2=2$. Thus $b_3=2-b_2=0$, and $b_1+b_3+b_4=0<2$,  a contradiction. Thus $u\notin I^{(2)}I^{(4)}$.  

Next, assume by contradiction that  $u=vw$ for some monomials $v,w\in I^{(3)}$. As before, let $v={\bf x}^{\bf b}$ and $w={\bf x}^{\bf c}$. Since $B\neq\emptyset$ and	$(G_{B})^c$ is not bipartite,  $G_{B}$ has an induced subgraph of the form $C_{2r+1}^c$. We
may assume that $C_{2r+1}$ is the cycle $5,6,\ldots,2r+5$. 
For any $5\leq i<2r+5$  since $\{i,i+1\}\notin E(G)$,  we have $b_i+b_{i+1}\geq 3$ and $c_i+c_{i+1}\geq 3$. On the other hand, $b_i+c_i=b_{i+1}+c_{i+1}=3$,  which gives  
	$b_i+b_{i+1}=c_i+c_{i+1}=3$ for any $5\leq i<2r+5$. This gives
	$b_i=b_{i+2}$ and $c_i=c_{i+2}$ for any $5\leq i<2r+4$. In particular,  
	$b_5=b_{2r+5}$ and $c_5=c_{2r+5}$. Since $\{5,2r+5\}\notin E(G)$, we have $2b_5=b_5+b_{2r+5}\geq 3$ and $2c_5=c_5+c_{2r+5}\geq 3$. Hence, $b_5\geq 2$ and $c_5\geq 2$, which contradicts to $b_5+c_5=3$. Thus $u\notin (I^{(3)})^2$. So we have shown that   $u={\textbf x}_A^2{\textbf x}_B^3{\textbf x}_C^4t^6$ is a minimal generator of $\mathcal{R}_{s}(I_{c}(G))$.
	
Next, assume that condition (ii) holds. We may assume that $A=[3]$  and $C_{2r+1}$ is the cycle $4,5,\ldots,2r+4$, where $C_{2r+1}^c$ is an induced subgraph of $G_{B}$.	 
We show that $u\notin I^{(2)}I^{(4)}+(I^{(3)})^2$.
By an argument identical to the one used in the case (i), we have $u\notin (I^{(3)})^2$.
	Now, by contradiction assume that $u\in I^{(2)}I^{(4)}$. Then $u=vw$ for some monomials $v\in I^{(2)}$ and $w\in I^{(4)}$. Let $v={\bf x}^{\bf b}$ and $w={\bf x}^{\bf c}$. We have $b_1+b_2+b_3=2$ and $c_1+c_2+c_3=4$. So $b_i=0$ for some $i\in[3]$. We may assume that $b_1=0$ and $b_2\geq 1$. Then $b_2+b_3=2$, $c_1=2$, $c_2+c_3=2$ and $c_2\leq 1$. Then either $b_2=b_3=1$ or $b_2=2$ and $b_3=0$. First we consider the case that $b_2=b_3=1$. Then $c_2=c_3=1$.
	Since $\{1,2,3\}$ is not split-dominating, there exists $p\in S$ such that $p\notin N_G(1)\cup[N_G(2)\cap N_G(3)]$. By the choice of $p$ we have $b_p=b_p+b_1\geq 2$. Thus $c_p=4-b_p\leq 2$ and so $c_2+c_p\leq 3$ and $c_3+c_p\leq 3$. This together with $p\notin N_G(2)\cap N_G(3)$ gives a contradiction to $w\in I^{(4)}$. Now, consider the case that $b_2=2$ and $b_3=0$.
	Since $\{1,2,3\}$ is not split-dominating, there exists $q\in S$ such that $q\notin N_G(2)\cup[N_G(1)\cap N_G(3)]$. So $b_q=b_q+b_1=b_q+b_3\geq 2$. Thus $c_q=4-b_q\leq 2$ and $c_2+c_q=c_q\leq 2$. This together with $q\notin N_G(2)$ contradicts to $w\in I^{(4)}$, which concludes the proof.
\end{proof}

The following corollary is an immediate consequence of Theorem~\ref{degree-six} and Theorem~\ref{symbolic Rees}. Recall that for graphs $G$ and $H$ on disjoint vertex sets, the {\em join} of $G$ and $H$ is the graph with the vertex set $V(G)\cup V(H)$ and the edge set $E(G)\cup E(H)\cup\{\{i,j\}:\, i\in V(G),\, j\in V(H)\}$ and is denoted by $G\vee H$. By $K_n$, $C_n$, respectively $P_n$, we denote a complete graph, a cycle graph, respectively a path graph, on $n$ vertices.

\begin{cor}\label{no-degree-six}
	Let $G$ be a graph. Then the graded $S$-algebra $\mathcal{R}_{s}(I_{c}(G))$ is generated in degree at most 4 if and only if 
	\begin{enumerate}
		\item [\textup{(i)}] $G$ is $(K_4\vee C_{2r+1}^c)$-free, and
		\item [\textup{(ii)}] For any induced subgraph $K_3\vee C_{2r+1}^c$ in $G$, the corresponding triangle $K_3$ is a split-dominating triangle of $G$.
	\end{enumerate}     
	 In particular, if $G$ is  $(K_3\vee C_{2r+1}^c)$-free, then $\mathcal{R}_{s}(I_{c}(G))$ is generated in degree at most 4. 
\end{cor}

An edge $\{i,j\}$ of a graph $G$ is called a {\em dominating edge} of $G$, if $\{i,j\}$ is a dominating set of $G$, that is $[n]=N_G(i)\cup N_G(j)$.

\begin{thm}\label{degree-four} 
		Let $G$ be a graph on $[n]$. Then $ut^4$ is a minimal generator of  $\mathcal{R}_{s}(I_{c}(G))$ if and only if $u=x_rx_s{\bf x}_B^2{\bf x}_C^3$, where $\{r,s\}\in E(G)$ is not a dominating edge of $G$, $B=N_G(r)\cap N_G(s)\neq \emptyset$, $B$ is not a clique of $G$ and $C=[n]\setminus (\{r,s\}\cup B)$.
\end{thm}

\begin{proof}
Let $ut^4$ be a minimal generator of  $\mathcal{R}_{s}(I_{c}(G))$. Write $u={\bf x}^{\bf a}$, and let $\alpha,\beta$ and $\gamma$ be as in the proof of Theorem \ref{symbolic Rees}. 	
Considering the proof of Theorem~\ref{symbolic Rees},   
the minimality of $ut^4$ implies that condition (iii) of the proof happens, that is $\alpha+\beta+\gamma=k=4$,  $\alpha=\beta<\gamma$ and $\gamma\leq 2\alpha$. Then we have $\alpha=\beta=1$ and $\gamma=2$. Let $\alpha=a_r$ and $\beta=a_s$. Since $a_r+a_s=2<4$, we have $\{r,s\}\in E(G)$. Moreover, for any $\ell\neq r,s$, we have $a_\ell\geq 4-(a_r+a_s)=2$. Set  $B=\{i:\, a_i=2\}$, and $C=[n]\setminus (\{r,s\}\cup B)$. The minimality of $ut^4$ implies that $a_i=3$ for any $i\in C$. Thus  
$u=x_rx_s{\bf x}_B^2{\bf x}_C^3$. Since $\gamma=2$, we have $B\neq \emptyset$. For any $i\in B$, from the equalities $a_i+a_r=3$ and $a_i+a_s=3$ and that $u\in I^{(4)}$ we obtain $i\in N_G(r)\cap N_G(s)$. Conversely, for any $i\in N_G(r)\cap N_G(s)$, if $i\notin B$, then $a_i=3$ and $u/x_i\in I^{(4)}$, which contradicts to the minimality of $ut^4$. Thus $B=N_G(r)\cap N_G(s)$. By contradiction assume that $B$ is a clique. Then $ut^4=({\bf x}_B{\bf x}_C)t.(x_rx_s{\bf x}_B{\bf x}_C^2)t^3=({\bf x}_{[n]\setminus\{r,s\}})t.(x_rx_s{\bf x}_B{\bf x}_C^2)t^3\in It.I^{(3)}t^3$, contradicting to the minimality of $ut^4$. Finally, we need to show that $\{r,s\}$ is not a dominating edge of $G$. Suppose on the contrary that $[n]=N_G(r)\cup N_G(s)$. We set $C_1=C\cap N_G(r)$ and $C_2=C\setminus C_1$. Our assumption implies that $C_2\subseteq N_G(s)$.
Then $ut^4=(x_r{\bf x}_B{\bf x}_{C_1}^2{\bf x}_{C_2})t^2.(x_s{\bf x}_B{\bf x}_{C_1}{\bf x}_{C_2}^2)t^2\in I^{(2)}t^2.I^{(2)}t^2$, a contradiction. Thus  $\{r,s\}$ is not a dominating edge, as claimed.

Conversely, let $u=x_rx_s{\bf x}_B^2{\bf x}_C^3$, where $\{r,s\}\in E(G)$ is not a dominating edge of $G$, $B=N_G(r)\cap N_G(s)\neq \emptyset$ is not a clique of $G$ and $C=[n]\setminus (\{r,s\}\cup B)$. Then it is easily seen that $u\in I^{(4)}$.  We show that $ut^4$ is a minimal generator of $\mathcal{R}_{s}(I_{c}(G))$. To this aim we need to show that $u\notin \m I^{(4)}+I.I^{(3)}+ (I^{(2)})^2$. It is straightforward to check that $u\notin \m I^{(4)}$. By contradiction assume that $u=vw$ for monomials $v\in I$ and $w\in I^{(3)}$. Write $v={\bf x}^{\bf b}$ and $w={\bf x}^{\bf c}$.
Without loss of generality we may assume that $\{r,s\}=\{1,2\}$. So $b_1+c_1=b_2+c_2=1$. Since $B$ is not a clique, there exist $j_1,j_2\in B$ such that $\{j_1,j_2\}\notin E(G)$. This yields that $c_{j_1}+c_{j_2}\geq 3$. So we may assume that $c_{j_1}\geq 2$. Since $j_1\in B$, we obtain $c_{j_1}=2$, and so $b_{j_1}=0$. Therefore, $b_1+b_2=b_1+b_2+b_{j_1}\geq 1$. Without loss of generality assume that $b_1\geq 1$. Then $b_1=1$ and $c_1=0$. 
Thus $c_2\geq 3-(c_1+c_{j_1})=1$ and so  $c_2=1$ and $b_2=0$. 
Since $\{1,2\}\in E(G)$ is not a dominating edge of $G$, there exists $p\in C$ such that $\{p,1\},\{p,2\}\notin E(G)$.  It follows that $b_p=b_p+b_2\geq 1$. Therefore, $c_p=3-b_p\leq 2$. We conclude that $c_p+c_1\leq 2$, which contradicts to $\{p,1\}\notin E(G)$. Thus $u\notin I.I^{(3)}$.

Finally, we show that $u\notin (I^{(2)})^2$. Suppose on the contrary that $u=vw$ for some monomials $v,w\in I^{(2)}$. Write $v={\bf x}^{\bf b}$ and $w={\bf x}^{\bf c}$. We have $b_1+c_1=b_2+c_2=1$. If $b_1=b_2=0$, then $b_i=b_i+b_1+b_2\geq 2$ for any $i\in B$. Hence, $c_i=0$ for any $i\in B$. Since $B\neq \emptyset$ is not a clique, there exist $j_1,j_2\in B$ such that $\{j_1,j_2\}\notin E(G)$. But $c_{j_1}+c_{j_2}=0$, which is a contradiction. Similarly, if  $c_1=c_2=0$, the same argument gives a contradiction. So we may assume that $b_1=1$ and $b_2=0$. Thus $c_1=0$. By assumption, there exists $p\in C$ such that $\{p,1\},\{p,2\}\notin E(G)$.
This yields that $b_p=b_p+b_2\geq 2$ and 
$c_p=c_p+c_1\geq 2$. Thus $3=b_p+c_p\geq 4$, a contradiction. So $u\notin (I^{(2)})^2$, as desired. This completes the proof. 
\end{proof}

Recall that a graph $G$ is called a {\em diamond} if it is isomorphic to $K_{4}-e$, for some $e\in E(K_4)$. The shared edge of the two triangles in a diamond is called the {\em central edge} of the diamond. Moreover, $G$ is said to be {\it diamond-free} if it contains no induced subgraph which is a diamond.  

\begin{cor}\label{no-degree-four}
Let $G$ be a graph. Then the graded $S$-algebra $\mathcal{R}_{s}(I_{c}(G))$ has no minimal generator of degree $4$ if and only if for any induced diamond $D$ in $G$, the central edge of $D$ is a dominating edge of $G$.     
\end{cor}

\begin{proof}
 By Theorem~\ref{degree-four}, any possible minimal generator of $\mathcal{R}_{s}(I_{c}(G))$ of degree $4$ is of the form $ut^4=x_rx_s{\bf x}_B^2{\bf x}_C^3t^4$, where $\{r,s\}$ is not a dominating edge of $G$, $B=N_G(r)\cap N_G(s)\neq \emptyset$, and   $G_{B}$ is not a complete graph. This means that $B$ contains two non-adjacent vertices $j_1,j_2$, and so the induced subgraph of $G$ on $\{r,s,j_1,j_2\}$ is a diamond with the central edge $\{r,s\}$. 
Now, assume that the central edge of any induced diamond in $G$ is a dominating edge of $G$. Then no monomial of the form $ut^4=x_rx_s{\bf x}_B^2{\bf x}_C^3t^4$, with the conditions as above, is a minimal generator of $\mathcal{R}_{s}(I_{c}(G))$. So  $\mathcal{R}_{s}(I_{c}(G))$  has no generator of degree $4$. 

Conversely, assume that $G$ has an induced diamond $D=\{i_1,i_2,i_3,i_4\}$ whose central edge $\{i_1,i_2\}$ is not dominating. Then $i_3,i_4\in B=N_G(i_1)\cap N_G(i_2)$ and $\{i_3,i_4\}\notin E(G)$. So $B$ is not a clique. Thus by Theorem~\ref{degree-four}, the monomial $ut^4=x_{i_1}x_{i_2}{\bf x}_B^2{\bf x}_C^3t^4$ with $C=[n]\setminus (\{i_1,i_2\}\cup B)$ is a  minimal generator of $\mathcal{R}_{s}(I_{c}(G))$ of degree $4$.
\end{proof}

\begin{cor}\label{diamond-free}
Let $G$ be a diamond-free graph. Then the graded $S$-algebra $\mathcal{R}_{s}(I_{c}(G))$ is generated in degree at most 3. More precisely, we have $$\mathcal{R}_{s}(I_{c}(G))=S[I_{c}(G)t,I_{c}(G)^{(2)}t^2, I_{c}(G)^{(3)}t^3].$$
\end{cor}

\begin{proof}
It is easily seen that if $G$ is diamond-free, then it is $(K_3\vee C_{2r+1}^c)$-free. So by Corollary~\ref{no-degree-six}, the algebra $\mathcal{R}_{s}(I_{c}(G))$ has no minimal generator of degree $6$. Moreover, Corollary~\ref{no-degree-four} shows that $\mathcal{R}_{s}(I_{c}(G))$ admits no minimal generator of degree $4$ either. Thus by Theorem~\ref{symbolic Rees}, $\mathcal{R}_{s}(I_{c}(G))$ is generated in degree at most 3. 
\end{proof}

A vertex $i$ in a graph $G$ is called a {\em cone} vertex, if $N_G[i]=V(G)$. 

\begin{thm}\label{degree-three} 
		Let $G$ be a graph on $[n]$. Then $ut^3$ is a minimal generator of  $\mathcal{R}_{s}(I_{c}(G))$ if and only if $u={\bf x}_A{\bf x}_B^2$, where $A$ is a maximal clique of $G$  and $B=[n]\setminus A$ such that one of the following conditions holds:
\begin{enumerate}
    \item[\textup{(i)}] $|A|\geq 4$.
    \item[\textup{(ii)}] $|A|=3$ and $A$ contains no cone vertex. 
\end{enumerate}      
\end{thm}

\begin{proof}
Let $ut^3$ be be a minimal generator of  $\mathcal{R}_{s}(I_{c}(G))$. Write $u={\bf x}^{\bf a}$, and let $\alpha,\beta$ and $\gamma$ be as in the proof of Theorem \ref{symbolic Rees}. 	
Considering the proof of Theorem~\ref{symbolic Rees},   
the minimality of $ut^3$ implies that condition (iv) of the proof happens, that is $\alpha+\beta+\gamma=k=3$,  and $\alpha=\beta=\gamma$. So we obtain $\alpha=\beta=\gamma=1$. We set $A=\{i:\, a_i=1\}$ and $B=[n]\setminus A$. From $u\in I^{(3)}$ it is easily seen that  $A$ is a clique of $G$, and since $\alpha=\beta=\gamma=1$, we have $|A|\geq 3$. Suppose on the contrary that $A$ is not a maximal clique. Then there exists $j\in B$ such that $A\cup\{j\}$ is a clique. We have $a_j\geq 2$. Since $\{i,j\}\in E(G)$ for any $i\in A$, it follows that $u/x_j\in I^{(3)}$. This contradicts to the minimality of $ut^3$. Hence, $A$ is a maximal clique, as desired. The minimality of $ut^k$ implies that $a_j=2$ for any $j\in B$. So $u={\bf x}_A{\bf x}_B^2$.
If $|A|\geq 4$, then (i) holds and we are done. Now, suppose $|A|=3$. Suppose on the contrary that $i_1\in A$ is a cone. Let $A\setminus \{i_1\}=\{i_2,i_3\}$. Then $$ut^3=(({\bf x}_A/x_{i_2}x_{i_3}){\bf x}_B)t.(x_{i_2}x_{i_3}{\bf x}_B)t^2=({\bf x}_{[n]\setminus \{i_2,i_3\}})t.(x_{i_2}x_{i_3}{\bf x}_B)t^2\in It.I^{(2)}t^2,$$
which contradicts to the minimality of $ut^3$.

Conversely, consider a monomial $u={\bf x}_A{\bf x}_B^2$, where $A$ is a maximal clique of $G$ satisfying (i) or (ii) and $B=[n]\setminus A$. Then it is easily seen that $u\in I^{(3)}$.  We need to show that $u\notin \m I^{(3)}+I.I^{(2)}$. For any $i\in A$, there exist distinct $j,r\in A\setminus\{i\}$ such that $a_i+a_j+a_r=3$. Hence, $u/x_i\notin  I^{(3)}$. Now, consider $i\in B$. Since $A$ is a maximal clique, there exists $j\in A$ such that $\{i,j\}\notin E(G)$. Since $a_i+a_j=3$, we obtain $u/x_i\notin  I^{(3)}$. Hence, $u\notin \m I^{(3)}$.    

Suppose now by contradiction that $u\in I.I^{(2)}$ and let $u=vw$ for some monomials $v\in I$ and $w\in I^{(2)}$. We may assume that $ A=[m]$ for some $m\geq 3$. Write $v={\bf x}^{\bf b}$ and $w={\bf x}^{\bf c}$. Then $b_i+c_i=1$ for any $i\in A$. Since $c_1+c_2+c_3\geq 2$, and $c_1,c_2,c_3\leq 1$, we may assume that $c_1=c_2=1$. This implies that $b_1=b_2=0$, $b_i=1$, and $c_i=0$ for any $3\leq i\leq m$. If $m\geq 4$, then $c_1+c_3+c_4=c_1=1<2$, which contradicts to  $w\in I^{(2)}$. Suppose now that $m=3$. Then by (ii), $A$ contains no cone. 
For any $i\in B$ we have $b_i=b_i+b_1+b_2\geq 1$, and hence $c_i=2-b_i\leq 1$. Since $3$ is not a cone, there exists $i\in B$ such that $\{3,i\}\notin E(G)$. But $c_3+c_i=c_i\leq 1$, which contradicts to $w\in I^{(2)}$. This concludes the proof.   
\end{proof}

The \emph{clique number} of a graph $G$, denoted $\omega(G)$, is the size of the largest clique in $G$. The following corollary follows from Theorem~\ref{degree-three}. Observe that when $\omega(G)=3$, then  $G$ has a cone vertex if and only if every triangle of $G$ contains a cone vertex. 

\begin{cor}\label{no-degree-three}
Let $G$ be a graph. Then the graded $S$-algebra $\mathcal{R}_{s}(I_{c}(G))$ has no minimal generator of degree $3$ if and only if $G$ is  triangle-free, or $\omega(G)=3$ and $G$ has a cone vertex.
\end{cor}

The following result characterizes graphs $G$ for which $\mathcal{R}_{s}(I_{c}(G))$ is generated in degree at most $2$. 
 
\begin{cor}\label{degree-at-most-two}
Let $G$ be a graph. Then $\mathcal{R}_{s}(I_{c}(G))$ is generated in degree at most $2$ if and only if either $G$ is triangle-free, or $\omega(G)=3$ and $G$ has a cone vertex.
\end{cor}

\begin{proof}
If $G$ is  triangle-free, then it is $(K_3\vee C_{2r+1}^c)$-free and diamond-free. So by Corollary~\ref{no-degree-six}, Corollary~\ref{diamond-free} and  Corollary~\ref{no-degree-three}, $\mathcal{R}_{s}(I_{c}(G))$ is generated in degree at most $2$.

Now, let $\omega(G)=3$, and let $s$ be a cone vertex of $G$. Then by Corollary~\ref{no-degree-six}, $\mathcal{R}_{s}(I_{c}(G))$ has no minimal generator of degree $6$. 
Consider an induced diamond $D$ in $G$ with $V(D)=\{i_1,i_2,i_3,i_4\}$, with the central edge $\{i_1,i_2\}$. Since $\{i_3,i_4\}\notin E(G)$, we have $s\neq i_3,i_4$. If $s\notin \{i_1,i_2\}$, then $i_1,i_2,i_3,s$ form a clique in $G$, which contradicts to $\omega(G)=3$. Hence, $s\in \{i_1,i_2\}$. This means that either $i_1$ or $i_2$ is a cone, yielding that $\{i_1,i_2\}$ is a dominating edge of $G$. So we have shown that for any induced diamond $D$, the central edge of $D$ is a dominating edge of $G$. 
So by Corollary~\ref{no-degree-four}, $\mathcal{R}_{s}(I_{c}(G))$ has no minimal generator of degree $4$. Finally, by Corollary~\ref{no-degree-three}, $\mathcal{R}_{s}(I_{c}(G))$ has no minimal generator of degree $3$. Thus $\mathcal{R}_{s}(I_{c}(G))$ is generated in degree at most $2$.

The converse statement follows from  Corollary~\ref{no-degree-three}.
\end{proof}

One should mention that when $G$ is triangle-free, then 
$\mathcal{R}_{s}(I_{c}(G))$ is just the symbolic Rees algebra of a cover ideal, and it is known by \cite[Theorem 5.1]{HHT} that such algebra is generated in degree at most $2$.

The following theorem determines the generators of $\mathcal{R}_{s}(I_{c}(G))$ of degree $2$. Recall that the \emph{matching number} of a graph $G$ is the maximum size of a set of pairwise disjoint edges in $G$.

\begin{thm}\label{degree-two} 
		Let $G$ be a graph on $[n]$. Then $ut^2$ is a minimal generator of  $\mathcal{R}_{s}(I_{c}(G))$ if and only if $u$ is of one of the following  forms:
\begin{enumerate}
    \item[\textup{(i)}]  $u={\bf x}_A{\bf x}_B^2$, where $A=N_G(i)$ for some vertex $i$ with $\deg_G(i)\geq 3$ and $B=[n]\setminus N_G[i]$.  
    \item[\textup{(ii)}] $u={\bf x}_{[n]}$, if $G$ has no cone, and if the matching number of $G$ is one, whenever $n=4$.   
\end{enumerate}        
\end{thm}

\begin{proof}
First assume that $ut^2$ is a minimal generator of  $\mathcal{R}_{s}(I_{c}(G))$, and write $u={\bf x}^{\bf a}$. Since $a_i+a_j+a_\ell\geq 2$ for any distinct integers $i$, $j$ and $\ell$, we have $a_i=0$ for at most two values of $i$. If there exist distinct $i$ and $j$ such that $a_i=a_j=0$, then $a_\ell\geq 2$ for any $\ell\neq i,j$. Hence $({\bf x}_{[n]}/x_ix_j)^2t^2\in I^2t^2$ divides $ut^2$, which contradicts to the minimality of $ut^2$. Now, assume that $a_i=0$ for precisely one index $i$. We show that $u$ is of the form described in (i). Set $A=\{j\in [n]:\, a_j=1\}$ and $B=\{j\in [n]:\, a_j\geq 2\}$. Since  $ut^2$ is a minimal generator, we have $a_j=2$ for any $j\in B$. Hence,  $u={\bf x}_A{\bf x}_B^2$. For any $j\notin N_G[i]$, we have $a_j=a_j+a_i\geq 2$, and so $j\in B$. Conversely, the minimality of $ut^2$ implies that $B\subseteq [n]\setminus N_G[i]$, and thus $B=[n]\setminus N_G[i]$.
Moreover, for any $j\in A$ since $a_i+a_j=1<2$ we have $j\in N_G(i)$. The minimality of $ut^2$ implies that $N_G(i)\subseteq A$ and so $A=N_G(i)$. 
By contradiction suppose that  $\deg_G(i)\leq 2$. If $\deg_G(i)=1$, say $N_G(i)=\{r\}$, then $(u/x_r)t^2=({\bf x}_{[n]\setminus  \{i,r\}})^2t^2\in I^2t^2$ divides $ut^2$, a contradiction. Let $\deg_G(i)=2$, and $N_G(i)=\{r,s\}$. Then $ut^2=x_rx_s.({\bf x}_{[n]\setminus \{i,r,s\}})^2t^2=({\bf x}_{[n]\setminus  \{i,r\}})t.({\bf x}_{[n]\setminus  \{i,s\}})t\in I^2t^2$, a contradiction. Hence, $\deg_G(i)\geq 3$.   
Thus $u$ is of the form described in (i).  
Now, assume that $a_i\neq 0$ for any $i\in[n]$. From the minimality of $ut^2$, it follows that $a_i=1$ for all $i$, and hence $u={\bf x}_{[n]}$. If $G$ has a cone vertex $i$, then $N_G(i)=[n]\setminus \{i\}$ and it is easily seen that ${\bf x}_{[n]\setminus \{i\}}\in I^{(2)}$ divides $u$, which means that $ut^2$ is not a minimal generator of $\mathcal{R}_{s}(I_{c}(G))$. Moreover, one can see that if $n=4$ and $G$ contains two disjoint edges, then $ut^2={\bf x}_{[n]}t^2\in (It)^2$ is not a minimal generator. This shows that (ii) holds for $u$.    

Conversely, let $u$ satisfies $(i)$. Then
it is straightforward to see that $u\in I^{(2)}$ and $u/x_j\notin I^{(2)}$ for any $j\neq i$. By contradiction, if $u\in I^2$, then ${\bf x}_{[n]\setminus \{r,s\}}{\bf x}_{[n]\setminus \{j,\ell\}}$ divides $u={\bf x}_{[n]\setminus \{i\}}{\bf x}_{[n]\setminus N_G[i]}$ for some edges $\{r,s\},\{j,\ell\}\in E(G)$. Comparing the degrees of both monomials it follows that $|N_G[i]|\leq 3$, which means $\deg_G(i)\leq 2$, a contradiction to the assumption in (i). So $u\notin I^2$, and hence $ut^2$ is minimal as desired.  
Finally, let $u$ satisfies (ii). Since $G$ has no cone, for any vertex $i$, we have $\{i,j\}\notin E(G)$ for some $j\neq i$, and so $u/x_i\notin I^{(2)}$. Moreover, since $G$ has no cone we have $n\geq 4$, and one can see that $u\in I^2$ if and only if $n=4$ and $G$ contains two disjoint edges. This shows that $ut^2$ is a minimal generator.  
\end{proof}

As a corollary of Theorem~\ref{degree-two}, we recover~\cite[Theorem 2.9]{rs2025} (see, also, \cite[Corollary 4.1]{f2026}).

\begin{cor}\label{equlaityPowes}
Let $G$ be a graph on $[n]$. Then $I_c(G)^{(k)}=I_c(G)^k$ for all $k\geq 1$ if and only if $G$ is either $K_2$, $K_3$, $P_3$, $2K_2$, $P_4$ or $C_4$.    
\end{cor}

\begin{proof}
We have $I_c(G)^{(k)}=I_c(G)^k$ for all $k\geq 1$ if and only if  $\mathcal{R}_{s}(I_{c}(G))=\mathcal{R}(I_{c}(G))$, that is 
$\mathcal{R}_{s}(I_{c}(G))$ is standard graded. By Corollary~\ref{degree-at-most-two} and Theorem~\ref{degree-two} we conclude that if $\mathcal{R}_{s}(I_{c}(G))$ is generated in degree one, then $n\leq 4$. Indeed if $n\geq 5$, then either $G$ has no cone, or $G$ has a vertex $i$ with $\deg_G(i)\geq 3$. In both cases, $\mathcal{R}_{s}(I_{c}(G))$ has a minimal generator of degree $2$, as Theorem~\ref{degree-two} shows. Hence, $n\leq 4$. Now, let $n=4$. If $G$ has a cone, then $\deg_G(i)\geq 3$ for some  vertex $i$, and so $\mathcal{R}_{s}(I_{c}(G))$ has a minimal generator of degree $2$. So $G$ has no cone and
by Theorem~\ref{degree-two}, $G$ contains two disjoint edges. The only graphs on $4$ vertices with this property are $2K_2,P_4$ and $C_4$. For $n=2,3$, since we assume that $G$ has no isolated vertices, the only possible graphs are $K_2$, $K_3$, and $P_3$. The converse statement also follows from Theorem~\ref{degree-two} and Corollary~\ref{degree-at-most-two}. 
\end{proof}


\section{Cycles}\label{cycles}
In this section, we restrict our attention to cycle graphs. We establish the symbolic Rees algebra of $I_c(G)$, compute the graded Betti numbers of
$I_c(G)^{(k)}$, and prove that all symbolic powers of $I_c(G)^{(k)}$ have linear quotients. We also determine the depth function $k\mapsto\depth S/I_c(G)^{(k)}$ and compute the Hilbert series of the symbolic fiber cone of $I_c(G)$.

Let $G=C_n$ be the cycle graph on $n$ vertices. That is, $G$ is the graph with the vertex set $V(G)=[n]$ and the edge set $E(G)=\{\{1,2\},\{2,3\},\dots,\{n-1,n\},\{n,1\}\}$. Let $S=K[x_1,\dots,x_n]$. We put $I=I_c(G)$ and ${\bf x}=x_1x_2\cdots x_n$.

\begin{thm}\label{main-cycle}
    With the notation introduced, the following statements hold.
    \begin{enumerate}
        \item[\textup{(a)}] We have
        $$
        \mathcal{R}_s(I)=\begin{cases}
        \mathcal{R}(I),&\text{if}\ n=3,4,\\
        S[\{ut\}_{u\in\mathcal{G}(I)},x_1\cdots x_nt^2],&\text{if}\ n\ge5.
        \end{cases}
        $$
        \item[\textup{(b)}] We have
        $$
        \widehat{\alpha}(I)=\begin{cases}
            1,&\text{if}\ n=3,\\
            2,&\text{if}\ n=4,\\
            \frac{n}{2},&\text{if}\ n\ge5.
        \end{cases}
        $$
        \item[\textup{(c)}] $I^{(k)}$ has linear quotients for all $k\ge1$. In particular, it is componentwise linear.\medskip
        \item[\textup{(d)}] For all $n\ge5$ and all $k\ge2$, putting $m=\lfloor\frac{k}{2}\rfloor$, we have\smallskip
        \begin{enumerate}
            \item[\textup{(i)}] If $k=2m$ is even, setting $d_p=2m(n-2)-p(n-4)$ for $0\le p\le m-1$, then
            $$
            \beta_{i,j}(I^{(k)})=\begin{cases}
                1,&\text{if}\ i=0,\ j=mn,\\
                n\big[\binom{1}{i}+(2m-2p-1)\binom{2}{i}\big],&\text{if}\ 0\le p\le m-1,\ j=d_p+i,\\
                0,&\text{otherwise}.
            \end{cases}
            $$
            \item[\textup{(ii)}] If $k=2m+1$ is odd, setting $d_p=(2m+1)(n-2)-p(n-4)$ for $0\le p\le m-1$, then
            $$
            \beta_{i,j}(I^{(k)})=\begin{cases}
                \binom{0}{i}+(n-2)\binom{1}{i}+\binom{2}{i},&\text{if}\ j=mn+n-2+i,\\
                n\big[\binom{1}{i}+(2m-2p)\binom{2}{i}\big],&\text{if}\ 0\le p\le m-1,\ j=d_p+i,\\
                0,&\text{otherwise}.
            \end{cases}
            $$
        \end{enumerate}
        \item[\textup{(e)}] For all $n\geq 4$, we have $\depth S/I^{(k)}=\lim_{k\rightarrow\infty}\depth S/I^{(k)}=n-3$, for all $k\ge1$. In particular $\ell_s(I)=3$.\smallskip
        \item[\textup{(f)}] For all $n\geq5$, the Hilbert series of $\mathcal{F}_{s}(I)$ is of the form
        $$
        F(\mathcal{F}_{s}(I),z)=\dfrac{1+(n-2)z+z^2}{(1-z)^{3}(1+z)}.
        $$
    \end{enumerate}
\end{thm}
\begin{proof}
    (a) It follows from Corollary~\ref{degree-at-most-two} that 
    $\mathcal{R}_s(I)$ is generated in degree at most $2$. Moreover, it follows from Theorem~\ref{degree-two} that the only possible generator of degree $2$ is $x_1\cdots x_nt^2$ which is minimal precisely when $n\geq 5$.
    \smallskip

    (b) If $n\in\{3,4\}$, then all the ordinary and symbolic powers of $I$ coincide. In those cases, $\widehat{\alpha}(I)=\alpha(I)=n-2$. Now, let $n\ge5$. By part (a), for all $k\ge2$ we have
    \begin{equation}\label{eq:decompI_c(C_n)}
        I^{(k)}=\sum_{p=0}^{\lfloor\frac{k}{2}\rfloor}{\bf x}^pI^{k-2p}.
    \end{equation}
    Hence $\alpha(I^{(k)})=n\lfloor\frac{k}{2}\rfloor+(n-2)(k-2\lfloor\frac{k}{2}\rfloor)$, for all $k\ge2$. So, $\widehat{\alpha}(I)=\lim_{k\rightarrow\infty}\alpha(I^{(k)})/k=n/2$.

 \smallskip
 
    (c) By \cite[Corollary 3.2]{fmStanley-Reisner2025} (see also \cite[Theorem 2.2]{hqm2026} or \cite[Corollary 4.8]{fm2025}) $I^{(1)}=I$ has linear quotients. For $n\in\{3,4\}$, $I$ is polymatroidal and $I^{(k)}=I^k$ for all $k\ge1$, so that $I^k$ has linear quotients for all $k\ge1$. Now, let $n\ge5$ and $k\ge2$. When a positive integer $j$ exceeds $n$ we take its remainder modulo $n$, and we maintain this convention throughout the proof. Put $u_j={\bf x}/(x_jx_{j+1})$ for $j\in[n]$, with our convention we have $x_{n+1}=x_1$ and $u_{n+1}=u_1$. We claim that
    \begin{equation}\label{eq:G(I^(k))-C_n}
        \mathcal{G}(I^{(k)})=\{{\bf x}^pu_j^a u_{j+1}^{b}:\ a\ge1,\ p,b\ge0,\ 2p+a+b=k,\ j\in[n]\}\cup\begin{cases}
            \{{\bf x}^{k/2}\},&\textup{if}\ k\ \textup{is even},\\
            \phantom{,}\emptyset,&\textup{otherwise}.
        \end{cases}
    \end{equation}
    
    To this end, first note that $I$ is equigenerated and $\mathcal{G}(I)=\{u_1,\dots,u_n\}$. By formula (\ref{eq:decompI_c(C_n)}), if $u\in\mathcal{G}(I^k)$, then $u={\bf x}^pu_{j_1}\cdots u_{j_{k-2p}}$ for some $p$ and some indices $j_1,\dots,j_{k-2p}\in[n]$. If $k-2p=0$, then $k$ is even, $p=k/2$ and $u={\bf x}^{k/2}$. Suppose that $k-2p>0$. We claim that $\{j_1,\dots,j_{k-2p}\}\subseteq\{j,j+1\}$ for some $j\in[n]$. Suppose that this was not the case, say that $j_{k-2p}$ and $j_{k-2p-1}$ are distinct non-consecutive integers. Then ${\bf x}$ strictly divides $u_{j_{k-2p-1}}u_{j_{k-2p}}$. Hence
    $$
    {\bf x}^{p+1}u_{j_1}\cdots u_{j_{k-2(p+1)}}\in {\bf x}^{p+1}I^{k-2(p+1)}\subset I^{(k)}\ \ \textup{strictly divides}\ \ u,
    $$
    against the fact that $u$ is a minimal generator of $I^{(k)}$. This proves our claim and shows that the right-hand set in (\ref{eq:G(I^(k))-C_n}) is contained in the left-hand set in (\ref{eq:G(I^(k))-C_n}). To conclude that (\ref{eq:G(I^(k))-C_n}) is indeed an equality, we must show that each monomial appearing in the right-hand side of (\ref{eq:G(I^(k))-C_n}) is a minimal generator of $I^{(k)}$. It is clear that if $k$ is even, then ${\bf x}^{k/2}$ is a minimal generator of $I^{(k)}$ because its degree is the initial degree of $I^{(k)}$.
    
    Now, let $u={\bf x}^pu_j^au_{j+1}^b$ with $a\ge1$, $p,b\ge0$, $2p+a+b=k$ and $1\le j\le n$. By the first part of the argument, $u\in I^{(k)}$. To prove that $u$ is a minimal generator of $I^{(k)}$ it is enough to show that $u/x_i\notin I^{(k)}$ for all $i\in[n]$. To this end, up to relabeling, we may assume that $j=1$. Hence $u={\bf x}^{k-p}/(x_1^ax_2^{a+b}x_3^b)=x_1^{c_1}\cdots x_n^{c_n}$ with $c_1=k-p-a$, $c_2=k-p-(a+b)=p$, $c_3=k-p-b$ and $c_j=k-p$ for all $4\le j\le n$. Now we show that for all $x_i$ dividing $u$, we have that $u/x_i$ does not belong to $I^{(k)}$ by distinguishing all possible cases.

    $\mathbf{(i=1)}$ Then $u/x_1=x_1^{c_1-1}x_2^{c_2}\cdots x_n^{c_n}\notin I^{(k)}$ because $\{1,3\}\notin E(G)$ but $(c_1-1)+c_3=k-1$.

    $\mathbf{(i=2)}$ Then $u/x_2=x_1^{c_1}x_2^{c_21-1}x_3^{c_3}\cdots x_n^{c_n}\notin I^{(k)}$ because $\{2,4\}\notin E(G)$ but $(c_2-1)+c_4=k-1$.

    $\mathbf{(i=3)}$ This case is analogous to the case $i=1$, because $c_1+(c_3-1)=k-1$.

    $\mathbf{(4\le i\le n)}$ This case can be handled as the case $i=2$, because for any $4\le j\le n$ we have $a_j=a_4$ and so $c_2+(c_j-1)=c_2+(c_4-1)=k-1$.

    Now, fix $k\ge2$. Having established equation (\ref{eq:G(I^(k))-C_n}), we claim that $I^{(k)}$ has linear quotients with respect to the following order. We put $w_0={\bf x}^{k/2}$ is $k$ is even, and $w_{p,a,b,j}={\bf x}^pu_j^au_{j+1}^b$ with $a\ge1$, $p,b\ge0$, $2p+a+b=k$ and $1\le j\le n$. Then, if $k$ is even we declare $w_0>w_{p,a,b,j}$ for all admissible 4-uples $(p,a,b,j)$ satisfying the above side conditions. For arbitrary $k\ge2$, we declare $w_{q,c,d,h}>w_{p,a,b,j}$ if $q>p$ or else if $q=p$ then we consider the following order
    \begin{equation}\label{matC_n}
        \begin{matrix}
        {\bf x}^pu_1^{k-2p-1}u_2&>&{\bf x}^pu_1^{k-2p-2}u_2^2&>&\dots&>&{\bf x}^pu_1u_2^{k-2p-1}&>&\\
        {\bf x}^pu_2^{k-2p-1}u_3&>&{\bf x}^pu_2^{k-2p-2}u_3^2&>&\dots&>&{\bf x}^pu_2u_3^{k-2p-1}&>&\\
        &&&>&\dots&>&&&\\
        {\bf x}^pu_{n-1}^{k-2p-1}u_n&>&{\bf x}^pu_{n-1}^{k-2p-2}u_n^2&>&\dots&>&{\bf x}^pu_{n-1}u_n^{k-2p-1}&>&\\
        {\bf x}^pu_n^{k-2p-1}u_1&>&{\bf x}^pu_n^{k-2p-2}u_1^2&>&\dots&>&{\bf x}^pu_nu_1^{k-2p-1}&>&\\
        {\bf x}^pu_1^{k-2p}&>&{\bf x}^pu_2^{k-2p}&>&\dots&>&{\bf x}^pu_n^{k-2p}.
    \end{matrix}
    \end{equation}

    For a generator $w\in\mathcal{G}(I^{(k)})$ let $J_w$ be the monomial ideal generated by the monomials $v\in\mathcal{G}(I^{(k)}$ with $v>w$. We must prove that $J_w:w$ is generated by variables. If $k$ is even and $w=w_0={\bf x}^{k/2}$ this is clear. So we may assume that $w=w_{p,a,b,j}$ for some $4$-uple $(p,a,b,j)$ satisfying the side conditions in (\ref{eq:G(I^(k))-C_n}). Notice that $k-2p=a+b\ge1$. We claim that
    \begin{equation}\label{eq:colonsC_n}
        J_w:w=\begin{cases}
        (x_{j+1}),&\textup{if}\ b=1,\\
        (x_{j+1},x_{j+2}),&\textup{if}\ b\ge2,\\
        (x_{j},x_{j+1}),&\textup{if}\ b=0,\ k-2p\ge2,\\
        (x_{j+1}),&\textup{if}\ b=0,\ k-2p=1,\ j<n,\\
        (x_1,x_n),&\textup{if}\ b=0,\ k-2p=1,\ j=n.
    \end{cases}
    \end{equation}

    To prove the claim, we first notice that $w_{p,a,b,j}=x_j^{p+b}x_{j+1}^px_{j+2}^{p+a}(\prod_{i\ne j,j+1,j+2}x_i^{k-p})$. Now we distinguish the possible cases.\smallskip

    $\mathbf{(b=1)}$ Since $a\ge1$, $b=1$ and $k-2(p+1)=a+b-2$, the $4$-tuple $(p+1,a-1,b-1,j)$ satisfies the side conditions in (\ref{eq:G(I^(k))-C_n}). Hence $w_{p+1,a-1,b-1,j}\in\mathcal{G}(I^{(k)})$ and $w_{p+1,a-1,b-1,j}:w_{p,a,b,j}=x_{j+1}$. This shows that $(x_{j+1})\subseteq J_w:w$. Conversely, it is enough to show that for any $w'=w_{q,c,d,h}>w_{p,a,b,j}=w$ the variable $x_{j+1}$ divides $w':w$. If $q>p$, since the exponent of $x_{j+1}$ in $w'$ is at least $q$, while the exponent of $x_{j+1}$ in $w$ is precisely $p$, we see that $x_{j+1}$ divides $w':w$. Otherwise if $q=p$, since $b=1$ from (\ref{matC_n}) we see that $h<j$ and that the exponent of $x_{j+1}$ in $w'$ is either $k-p>p$ (if $h\ge j-2$) or $p+c>p$ (if $h=j-1$). This then shows that $x_{j+1}$ divides $w':w$. Finally, if $k$ is even, then it is immediately seen that $x_{j+1}$ divides ${\bf x}^{k/2}:w$.\smallskip

    $\mathbf{(b\ge2)}$ We have $w_{p+1,a-1,b-1,j}:w_{p,a,b,j}=x_{j+1}$ and $w_{p,a+1,b-1,j}:w_{p,a,b,j}=x_{j+2}$. Hence $(x_{j+1},x_{j+2})\subseteq J_w:w$. Conversely, it is enough to show that for any $w'=w_{q,c,d,h}>w_{p,a,b,j}=w$ then $x_{j+1}$ or $x_{j+2}$ divide $w':w$. If $q>p$ this is clear as argued in the previous case. Suppose that $q=p$. If $h<j$, then $x_{j+1}$ divides $w':w$ as argued in the previous case. If $h=j$, then $w'>w$ implies that $c>a$. Since $k-2p=a+b=c+d$, this then implies that $b>d$ and a quick computation then shows that $x_{j+2}$ divides $w':w$. Finally, if $k$ is even, then either $x_{j+1}$ or $x_{j+2}$ divides ${\bf x}^{k/2}:w$.\smallskip

    $\mathbf{(b=0,\ k-2p\ge2)}$ We have $w_{p,k-2p-1,1,j}:w=x_j$ and ${\bf x}^pu_{j-1}u_j^{k-2p-1}:w=x_{j+1}$. Hence $(x_j,x_{j+1})\subseteq J_w:w$. Conversely, let $w'=w_{q,c,d,h}>w$. If $q>p$, then both $x_j$ and $x_{j+1}$ divide $w':w$. If $q=p$, then either $d\ge1$ or $d=0$ and $h<j$. In both cases, one can easily see that either $x_j$ or $x_{j+1}$ divide $w':w$. Finally, if $k$ is even, then either $x_{j}$ or $x_{j+1}$ divides ${\bf x}^{k/2}:w$.\smallskip

    $\mathbf{(b=0,\ k-2p=1,\ j<n)}$ Then $k=2p+1$ is odd, and a direct computation shows that
    $$
    J_w:w=({\bf x}^pu_1,\dots,{\bf x}^pu_{j-1}):({\bf x}^pu_j)=(x_{j+1}),
    $$
    as desired.\smallskip

    $\mathbf{(b=0,\ k-2p=1,\ j=n)}$ Then $k=2p+1$ is odd, and a direct computation shows that
    $$
    J_w:w=({\bf x}^pu_1,\dots,{\bf x}^pu_{n-1}):({\bf x}^pu_n)=(x_{1},x_{n}),
    $$
    as claimed.\medskip

    (d) The given formulas in (i)-(ii) follow immediately from  \cite[Corollary 8.2.2, Exercise 8.8]{HHBook} and the computations in (\ref{eq:colonsC_n}).\smallskip

    (e) From \cite[Corollary 1.3]{hqm2026} or \cite[Corollary 2.4]{fm2025} and the fact that $C_{n}$ is a triangle free, we have $I^{\vee}=I(C_{n}^{c})$, where $I^{\vee}$ is the Alexander dual ideal of $I$. Then, it is known that $\reg I(C_{n}^{c})=3$, and thus, from \cite[Corollary 0.3]{t1999}, we get $\pd S/I=\reg I^{\vee}=3$. Hence, we obtain $\depth S/I=n-3$. For $k\geq 2$, the statement follows from (d).\smallskip
    
    (f) Note that $F(\mathcal{F}_{s}(I),z)=1+nz+\sum_{k\geq2}\mu(I^{(k)})z^{k}$ since $\mu(I)=n$. From (d) we have that
    $$
    \beta_{0}(I^{(k)})=\mu(I^{(k)})=\begin{cases}
        2n\binom{m+1}{2}+1=nm(m+1)+1,&\text{if}\ k=2m\ \text{is even},\\[7pt]
        n\big[\binom{m+1}{2}+\binom{m+2}{2}\big]=n(m+1)^{2},&\text{if}\ k=2m+1\ \text{is odd}.
    \end{cases}
    $$
    Hence,
    $$
    F(\mathcal{F}_{s}(I),z)=1+nz+\displaystyle\sum_{m\geq1}(nm(m+1)+1)z^{2m}+\displaystyle\sum_{m\geq1}n(m+1)^{2}z^{2m+1}.
    $$
    Using the standard identities $\sum_{i\geq1}\binom{i+1}{2}T^{i}=\frac{T}{(1-T)^{3}}$ and $\sum_{i\geq1}(i+1)^{2}T^{i}=\frac{1+T}{(1-T)^{3}}-1$, we see that 
    \begin{align*}
    F(\mathcal{F}_{s}(I),z)&=1+nz+\dfrac{2nz^{2}}{(1-z^{2})^{3}}+\dfrac{z^{2}}{1-z^{2}}+nz\left(\dfrac{1+z^{2}}{(1-z^{2})^{3}}-1\right) \\
    &=1+\dfrac{z^{2}}{1-z^{2}}+\dfrac{2nz^{2}}{(1-z^{2})^{3}}+\dfrac{nz(1+z^{2})}{(1-z^{2})^{3}} \\
    &=\dfrac{1}{1-z^{2}}+\dfrac{n(2z^{2}+z(1+z^{2}))}{(1-z^{2})^{3}} \\
    &=\dfrac{1+(n-2)z+z^2}{(1-z)^{3}(1+z)},
    \end{align*}
    which completes the proof. 
\end{proof}


\section{Complete multipartite graphs}\label{CompleteBipartiteG}

In this section we focus on complete multipartite graphs. We describe the symbolic Rees algebra $\mathcal{R}_s(I)$ of $I=I_c(G)$, then give a structural description of minimal $k$-covers (Lemma \ref{k-cover}), yielding an upper bound on $\mu(I^{(k)})$. We then determine $\lim_{k\to\infty}\depth S/I_c(G)^{(k)}$ together with a uniform lower bound for all $k\geq1$, and show that $I^{(k)}$ has linear quotients for every $k$ (Theorem \ref{LQ of complete multi}). This gives a complete depth formula when every part has size at least two (Proposition \ref{depth of complete multi}), which in particular shows that the depth function $k\mapsto\depth S/I_c(G)^{(k)}$ need not be non-increasing. Finally, we compute the Waldschmidt constant of $I$ in Proposition \ref{Waldschmit-constant}.

Throughout this section $G=K_{d_1,\dots,d_m}$ is a \textit{complete multipartite graph} with the vertex set $V(G)=\{x_{i,j}:i\in[m],j\in[d_i]\}$ and the edge set
$$
E(G)=\{\{x_{i,j},x_{p,q}\}:\ i\ne p,\ j\in[d_i],\ q\in[d_p]\}.
$$
Also $S=K[x_{i,j}:i\in[m],j\in[d_i]]$ and  $I=I_c(G)$. We put ${\bf x}_i=\prod_{j\in[d_i]}x_{i,j}$ for all $i\in[m]$ and ${\bf x}=\prod_{i\in[m]}{\bf x}_i$.
\begin{thm}\label{multipartitre-Rees}
    With the notation introduced, we have
        $$
        \mathcal{R}_s(I)=S[It,{\bf x}t^2,\left\{
\left(\frac{{\bf x}\cdot{\bf x}_i}{x_{i,j}^2}\right)t^2
\right\}_{\substack{i\in[m],j\in[d_i]}},\left\{
\frac{{\bf x}^2}{x_{1,j_1}\cdots x_{m,j_m}}t^3
\right\}_{\substack{j_s\in[d_s]
}}, \{u_{i_1,\ldots,i_r,j_1,\ldots,j_r}t^6\}_{r\geq 4}],
        $$
        where $u_{i_1,\ldots,i_r,j_1,\ldots,j_r}=0$, if $m\leq 4$ or $d_i\leq 2$ for all $i$, and otherwise $$u_{i_1,\ldots,i_r,j_1,\ldots,j_r}=\prod_{\ell=1}^r({\bf x}_{i_\ell}^4/x_{i_\ell,j_\ell}^2).\prod_{i\in[m]\setminus \{i_1,\ldots,i_r\}}{\bf x}_i^3,$$ for any pairwise distinct integers $i_1,\ldots,i_r\in [m]$ and $j_\ell\in [d_{i_\ell}]$ for all $1\leq\ell\leq r$.  
\end{thm}
\begin{proof}
    
    It is easily seen that all the monomials presented in the statement belong to $\mathcal{R}_s(I)$. Conversely, we will prove that any minimal generator $ut^k$ of $\mathcal{R}_s(I)$ has one of the presented forms. By Theorem~\ref{symbolic Rees} we have $k\leq 6$.
For $k=1$, there is nothing to prove because $I^{(1)}=I$. Let $k=2$. By Theorem~\ref{degree-two}, either $u={\bf x}$ or $u={\bf x}_A{\bf x}_B^2$, where $A=N_G(x_{i,j})={\bf x}/{\bf x}_i$ for some vertex $x_{i,j}$  and $B=V(G)\setminus N_G[x_{i,j}]={\bf x}_i/x_{i,j}$. Hence, $u=({\bf x}/{\bf x}_i)({\bf x}_i/x_{i,j})^2={\bf x}{\bf x}_i/x_{i,j}^2$, as desired. 
Now, let $k=3$. Since any maximal clique of $G$ has the vertex set of the form $x_{1,j_1},\ldots,x_{m,j_m}$,
the result follows from Theorem~\ref{degree-three}.

By Theorem~\ref{symbolic Rees} it remains to consider $k\in\{4,6\}$. Since any edge of $G$ is a dominating edge, by Theorem~\ref{degree-four} we conclude that $\mathcal{R}_s(I)$ has no minimal generator of degree $4$. So $k\neq 4$. Finally, assume that $k=6$. Notice that any triangle in $G$ is a split-dominating triangle. Thus by Theorem~\ref{degree-six}, $u={\bf x}_A^2{\bf x}_B^3{\bf x}_C^4$, where $G_{A}=K_r$ for some $r\geq 4$, $B=\{j\in[n]\setminus A:\, \{i,j\}\in E(G) \text{ for any } i\in A\}\neq\emptyset$,  $C=[n]\setminus (A\cup B)$ and $(G_{B})^c$ is not a bipartite graph. The description of $A$ implies that ${\bf x}_A=x_{i_1,j_1}\cdots x_{i_r,j_r}$ for pairwise distinct integers $i_1,\ldots,i_r\in[m]$ and $j_\ell\in [d_{i_\ell}]$. Also the description  of $B$ implies that ${\bf x}_B=\prod_{i\in[m]\setminus \{i_1,\ldots,i_r\}}{\bf x}_i$. Thus ${\bf x}_C=\prod_{\ell=1}^r({\bf x}_{i_\ell}/x_{i_\ell,j_\ell})$. Since  $B\neq\emptyset$ we have $m\geq 5$ and since $(G_{B})^c$ is not bipartite, we have $d_i\geq 3$ for some $i\in [m]\setminus \{i_1,\ldots,i_r\}$.  This shows that $u=u_{i_1,\ldots,i_r,j_1,\ldots,j_r}$, as desired.
\end{proof}

A {\em minimal $k$-cover} of $I$ is a vector $\textbf{a}=(a_1,\ldots,a_n)$ such that ${\bf x}^{\textbf{a}}$ is a minimal generator of $I^{(k)}$. 
It is easily seen that $\textbf{a}$  is a minimal $k$-cover of $I_c(G)$ if and only if
 
\begin{enumerate}
    \item[(\textup{i})]  
    $\begin{cases}
        a_{i,h}+a_{i,\ell}\ge k,&\textup{for all}\ i\in[m]\ \textup{and all}\ 1\le h<\ell\le d_i,\\
        a_{i_1,j_1}+a_{i_2,j_2}+a_{i_3,j_3}\ge k,&\textup{for all}\ 1\le i_1<i_2<i_3\le m\ \textup{and all}\ j_s\in[d_{i_s}].
    \end{cases}$
    \smallskip
    
    \item[(\textup{ii})]    For any $i$ and $j\in [d_i]$ with $a_{i,j}>0$, either there exists  some $\ell\in [d_i]$ with $\ell \neq j$ such that $a_{i,j}+a_{i,\ell}=k$ or there exist $i_2<i_3$ distinct from $i$ such that $a_{i,j}+a_{i_2,j_2}+a_{i_3,j_3}=k$,  for some $j_1\in[d_{i_1}]$ and $j_2\in[d_{i_2}]$.
    \end{enumerate}

\begin{lemma}\label{k-cover}
Let $G=K_{d_1,\dots,d_m}$ be a complete multipartite graph on $n\geq 3$ vertices, let $I=I_c(G)$, and let ${\bf a}$ be a minimal $k$-cover of $I$. Fix $i\in[m]$. Then the following statements hold.
\begin{enumerate}
   \item[\textup{(a)}] If $a_{i,j}\leq k/2$ for some $j$, then $a_{i,\ell}=k-a_{i,j}$ for any $\ell\in[d_i]\setminus j$.

   \item[\textup{(b)}] If $a_{i,j}>k/2$ for all $j$, then $a_{i,1}=a_{i,2}=\cdots=a_{i,d_i}$.
\end{enumerate}
Consequently, $$|\{(a_{i,1},\ldots, a_{i,d_{i}})\,\,: (a_{i,1},\ldots, a_{i,d_{i}})\mbox{ arises from a minimal }k\mbox{-cover}\}|\leq d_{i}(k+1),$$ 
and $\mu(I^{(k)})\leq (\prod_{i=1}^m d_i)(k+1)^m$.
\end{lemma}

\begin{proof}
(a) Without loss of generality assume that $a_{i,1}\leq k/2$. We have $a_{i,1}+a_{i,\ell}\geq k$, that is, $a_{i,\ell}\geq k-a_{i,1}$ for any $\ell\neq1$. We claim that $a_{i,\ell}=k-a_{i,1}$ for any $\ell\neq1$. Suppose to the contrary, that $a_{i,\ell}>k-a_{i,1}$ for some $\ell\neq1$. We show that any inequality in (i) containing $a_{i,\ell}$ is strict. This will give a contradiction to condition (ii). 
Let $h\in [d_i]\setminus \{\ell\}$. If $h=1$, then by our assumption, $a_{i,\ell}+a_{i,1}>k$. Now, assume that $h\neq 1,\ell$. Then we have $a_{i,\ell}+a_{i,h}>(k-a_{i,1})+(k-a_{i,1})=2(k-a_{i,1})\geq k$. Next, consider an inequality given in (i) of the form $a_{i,\ell}+b+c\geq k$. Then we also have $a_{i,1}+b+c\geq k$.
Hence, we get $b+c\geq k-a_{i,1}$, which implies $a_{i,\ell}+b+c>(k-a_{i,1})+(k-a_{i,1})=2(k-a_{i,1})\geq k$. Consequently, every inequality containing $a_{i,\ell}$ is strict, which is a contradiction.
\medskip

(b) Set $a_{i,r}=\max\{a_{i,j}: j\in [d_i]\}$. We have $a_{i,r}+a_{i,\ell}>k$ for every $\ell\neq r$. Hence, by (2), there exist $b,c$ such that $a_{i,r}+b+c=k$. We also have $a_{i,\ell}+b+c\geq k$, which implies that $a_{i,\ell}\geq k-b-c=a_{i,r}\geq a_{i,\ell}$. Thus $a_{i,\ell}=a_{i,r}$ for all $\ell$.
\end{proof}

\begin{thm}\label{limDepth}
Let $G=K_{d_1,\dots,d_m}$ be a complete multipartite graph on $n\geq 3$ vertices. Let $s=|\{i\,\,: d_{i}=1\}|$ and $t=|\{i\,\,: d_{i}\geq2\}|$. Then 
$$\lim_{k\rightarrow\infty}\depth S/I_{c}(G)^{(k)}=\begin{cases}
n-m-1, & \text{if } d_i\geq 2 \text{ for all } i, \\
n-t-\min\{3,s\}, & \text{otherwise}. 
\end{cases}$$ 
Moreover, for any  $k\geq 1$ we have  $$\depth S/I_{c}(G)^{(k)}\geq \begin{cases}
n-m-1, & \text{if } d_i\geq 2 \text{ for all } i,\\
n-t-\min\{3,s\}, & \text{otherwise}. 
\end{cases}$$ 
\end{thm}
\begin{proof}
Set $I=I_c(G)$. By Theorem~\ref{limit}, we have
$\lim_{k\rightarrow\infty}\depth S/I^{(k)}=n-\ell_s(I)$. So  it is enough to show that 
\begin{equation}\label{symbAnSp}
\ell_s(I)=\begin{cases}
m+1, & \text{if } d_i\geq 2 \text{ for all } i, \\
t+\min\{3,s\}, & \text{otherwise}. 
\end{cases}  
\end{equation}

Since $\mathcal{F}_s(I)$ is a finitely generated $K$-algebra, by \cite[Theorem 2.1]{HHT} there exists a positive integer $p$ such that the $p$th Veronese subalgebra $\mathcal{F}_s(I)^{(p)}$ of $\mathcal{F}_s(I)$ is a standard graded $K$-algebra. We set $A=\mathcal{F}_s(I)^{(p)}$. Then $A$ has a Hilbert polynomial, say $h_A$ of degree $\dim A-1$ such that $h_A(i)=\dim_K(\mathcal{F}_s(I)_{pi})$ for $i\gg 0$. Note that $\dim A=\dim\,\mathcal{F}_s(I)=\ell_s(I)$, since $\mathcal{F}_s(I)$ is a finitely generated module over $A$.
On the other hand, $\dim_K(\mathcal{F}_s(I)_{k})=\mu(I^{(k)})$ for all $k$, where $\mu(I^{(k)})$ denotes the number of minimal monomial generators of $I^{(k)}$.
Therefore, $$h_A(i)=\mu(I^{(pi)})=|\{\textbf{a}:\ \textbf{a}  \text{ is a minimal $pi$-cover of } I\}| \text{\ \ \  for } i\gg 0,$$ 
and $h_A$ is of degree $\ell_s(I)-1$.
Thus, in order to prove (\ref{symbAnSp}), we need to show that $$\deg(h_A)=\begin{cases}
m, & \text{if } d_i\geq 2 \text{ for all } i, \\
t+\min\{2,s-1\}, & \text{otherwise}. 
\end{cases}$$  

\textbf{Case 1}. Let $d_i\geq 2$ for all $i$. First we show that $\deg(h_A)\geq m$. 
Let $k$ be a positive integer such that $6|k$. For any $1\leq i\leq m$, choose an arbitrary integer $k/3<r_i\leq k/2$. We show that the vector $\textbf{a}$ defined as $$a_{i,j}=\begin{cases}
r_i, & \text{if } j=1, \\
k-r_i, & \text{if } 1<j\leq d_i. 
\end{cases}$$  
is a minimal $k$-cover of $I$. 
We have $a_{i,j}>k/3$ for all $i$ and $j$. Moreover, for distinct integers $j,\ell\in [d_i]$ we have either $a_{i,j}+a_{i,\ell}=k$ or $a_{i,j}+a_{i,\ell}=2k-2r_i\geq k$. So the vector $\textbf{a}$ satisfies (i). For arbitrary $i$ and $j$, we have $a_{i,j}+a_{i,1}=k$ for any $j\neq 1$, and hence $\textbf{a}$ satisfies (ii), as well. Thus, for any $k$ with $6|k$, each sequence of integers $r_1,\ldots,r_m$ with  $k/3<r_i\leq k/2$ gives a minimal $k$-cover of $I$.   
Thus $\mu(I^{(k)})\geq (k/6)^m$. This implies that $\deg(h_A)\geq m$.
Conversely, by Lemma~\ref{k-cover}, we have $\mu(I^{(k)})\leq (\prod_{i=1}^m d_i)(k+1)^m$. Consequently, we obtain the reverse inequality $\deg(h_A)\leq m$.  

\medskip

\textbf{Case 2}. Let $d_i=1$ for some $i$. First assume that $s=m$. Then $d_i=1$ for all $i$, and $G$ is the complete graph $K_n$. So $I=I_c(K_n)$ is the Stanley--Reisner ideal of a matroid. Thus using \cite[Theorem 2.1]{v2011} and Theorem \ref{min prime} we obtain $\depth S/I^{(k)}=\dim S/I^{(k)}=n-3$ for all $k\geq 1$.

Now, let $1\leq s<m$. We may assume that $d_{1}=\cdots=d_{s}=1$ and $d_{s+1},\ldots, d_{m}\geq2$.
First we show that $\deg(h_A)\geq t+\min\{2,s-1\}$. 
Let $k$ be a positive integer such that $6|k$. We distinguish the following cases.

\textbf{Subcase 2.1}.
Suppose that $s=1$.  
For any $2\leq i\leq m$, choose an arbitrary integer $\dfrac{k}{3}<r_i\leq \dfrac{k}{2}$. It is straightforward to see that the vector $\textbf{a}$ defined as $$a_{i,j}=\begin{cases}
r_i, & \text{if } 2\leq i\leq m, j=1, \\
k-r_i, & \text{if } 2\leq i\leq m, 1<j\leq d_i,\\
k-\min\{r_i+r_j: 2\leq i<j\leq m\}, & \text{if } i=1, j=1.
\end{cases}$$  
is a minimal $k$-cover of $I$. This shows that $\mu(I^{(k)})\geq (\dfrac{k}{6})^{m-1}$, and hence $\deg(h_A)\geq m-1=t=t+\min\{2,s-1\}$.   

\textbf{Subcase 2.2}. Suppose that $s=2$. 
Corresponding to every sequence of integers $\dfrac{k}{3}<r_2<r_3<\cdots<r_m\leq \dfrac{k}{2}$, we have the minimal $k$-cover $\textbf{a}$ defined as  
$$a_{i,j}=\begin{cases}
r_i, & \text{if } 2\leq i\leq m, j=1, \\
k-r_i, & \text{if } 3\leq i\leq m, 1<j\leq d_i,\\
k-(r_2+r_3): 2<j\leq m\}, & \text{if } i=1, j=1.
\end{cases}$$ 
Hence, $\mu(I^{(k)})\geq {k/6\choose m-1}$, which is a function in $k$ of degree $m-1$.
So we conclude that $\deg(h_A)\geq m-1=t+1=t+\min\{2,s-1\}$, as required.

\textbf{Subcase 2.3}. Suppose that $s\geq 3$. Corresponding to every sequence of integers $\dfrac{k}{3}<r_1<r_2<r_{s+1}\cdots<r_m\leq \dfrac{k}{2}$, there exists a minimal $k$-cover $\textbf{a}$ of $I$, defined as 
$$a_{i,j}=\begin{cases}
r_i, & \text{if } s+1\leq i\leq m, j=1, \\
k-r_i, & \text{if } s+1\leq i\leq m, 1<j\leq d_i,\\
r_i, & \text{if } 1\leq i\leq 2, j=1,\\
k-(r_1+r_2), & \text{if } i=3, j=1,\\
r_2, & \text{if } 3<i\leq s, j=1.
\end{cases}$$   
So we conclude that $\mu(I^{(k)})\geq {k/6\choose t+2}$, which is a function in $k$ of degree $t+2$.
Therefore, $\deg(h_A)\geq t+2=t+\min\{2,s-1\}$.

We now prove the converse inequality. It suffices to show that $\mu(I^{(k)})\leq c(k+1)^{t+\min\{s-1,2\}}$ for some $c>0$. By Lemma~\ref{k-cover}, for any $i\geq s+1$ we have
$$|\{(a_{i,1},\ldots, a_{i,d_{i}})\,\,: (a_{i,1},\ldots, a_{i,d_{i}})\mbox{ arises from a minimal }k\mbox{-cover}\}|\leq d_{i}(k+1).$$ Now, fix a vector $\textbf{a}'=(a_{s+1,1},\ldots, a_{s+1,d_{s+1}},\ldots, a_{m,1},\ldots, a_{m,d_{m}})$ arising from a minimal $k$-cover. An inequality containing exactly one of $a_{1,1},\ldots, a_{s,1}$ has the form $a_{r,1}+a_{i,j}+a_{\ell,h}\geq k$ with $i,\ell\geq s+1$. Since $a_{i,j}$ and $a_{\ell,h}$ are fixed, there exists $\alpha$ such that $a_{r,1}\geq\alpha$ for all $r$. Similarly, we see that there exists $\beta$ such that $a_{r,1}+a_{u,1}\geq\beta$ and $a_{r,1}+a_{u,1}+a_{v,1}\geq k$ for all distinct $r,u,v$. We first suppose that $s=1$. Then there exists $1<i<\ell$ such that $a_{1,1}+a_{i,j}+a_{\ell,h}=k$ for some $j$ and $h$. Thus $a_{1,1}$ is determined by the vector $\textbf{a}'$. So we obtain $$\mu(I^{(k)})\leq (\prod_{i=2}^md_{i})(k+1)^{t}=(\prod_{i=2}^md_{i})(k+1)^{t+\min\{s-1,2\}}.$$    
Next, suppose that $s=2$. We may assume that $a_{1,1}\leq a_{2,1}$. Fix $a_{1,1}$. Then by the minimality of the $k$-cover, $a_{2,1}$ is determined by ${\bf a}^{\prime}$ and $a_{1,1}$. Since $0\leq a_{1,1}\leq k$, there are at most $k+1$ choices after ordering, and hence, we obtain $\mu(I^{(k)})\leq (2\prod_{i=3}^md_{i})(k+1)^{t+1}=(2\prod_{i=3}^m d_{i})(k+1)^{t+\min\{s-1,2\}}$. Finally, suppose that $s\geq3$. We may assume that $a_{1,1}\leq a_{2,1}\leq\cdots\leq a_{s,1}$. Fix $a_{1,1}$ and $a_{2,1}$. Then by the minimality of $k$-cover, $a_{3,1},\ldots, a_{s,1}$ are determined by ${\bf a}^{\prime}$, $a_{1,1}$ and $a_{2,1}$. Since $0\leq a_{1,1}, a_{2,1}\leq k$, there are at most $(k+1)^{2}$ choices after ordering, and hence at most $s!(k+1)^{2}$ choices without ordering. Therefore, we obtain 
$\mu(I^{(k)})\leq (s!\prod_{i=s+1}^m d_{i})(k+1)^{t}(k+1)^{2}=(s!\prod_{i=s+1}^m d_{i})(k+1)^{t+\min\{s-1,2\}}$, which completes the proof. 
\end{proof}

\begin{thm}\label{LQ of complete multi}
Let $G=K_{d_1,\dots,d_m}$ be a complete multipartite graph on $n$ vertices, and let $I=I_c(G)$. Then $I^{(k)}$ has linear quotients. In particular, $I^{(k)}$ is componentwise linear. 
\end{thm}
\begin{proof}
If $k=1$, then by \cite[Corollary 3.2]{fmStanley-Reisner2025}, the assertion holds. So we may assume that $k\geq 2$. Moreover, if $m=2$, then $I^{(k)}=J(K_{d_1})^{(k)}\cap J(K_{d_2})^{(k)}=J(K_{d_1})^{(k)}J(K_{d_2})^{(k)}$, where $K_{d_i}$ denotes the complete graph on the vertex set $X_i=\{x_{i,\ell}: \ell\in [d_i]\}$. Since $J(K_{d_i})$ is the squarefree Veronese ideal $I_{d_i,d_i-1}$, by \cite[Corollary 3.3]{fmpoly2025} the ideal $J(K_{d_i})^{(k)}$ has linear quotients for $i=1,2$. Now, by \cite[Lemma 4.7(b)]{fmpoly2025}, $J(K_{d_1})^{(k)}J(K_{d_2})^{(k)}$ has linear quotients. So from now on we assume that $m\geq 3$ and $k\geq 2$.
Without loss of generality, let $d_1=\cdots=d_s=1$ and $d_i\geq 2$ for any $s+1\leq i\leq m$.		
For each
$i\in[m]$, let $\mathbf{x}_i=\prod_{\ell\in[d_i]}x_{i,\ell}$, and for any $j\in[d_i]$, and $0\le r\le k/2$, let
$w_{i,j,r}=
x_{i,j}^{\,r}(\mathbf{x}_i/x_{i,j})^{k-r}$. Notice that if $d_i=1$, then
$w_{i,1,r}=x_{i,1}^{\,r}$.
By Lemma~\ref{k-cover}, every minimal monomial generator of $I^{(k)}$
is of the form
\begin{equation}\label{generator}
	u=\prod_{i\in A}w_{i,j_i,r_i}
	\prod_{i\in[m]\setminus A}\mathbf{x}_i^{\,\alpha_i},
\end{equation}
where $A\subseteq[m]$, $0\le r_i\le k/2$, $\alpha_i>k/2$ for every
$i\in [m]\setminus A$, and $r_i+r_j+r_h\ge k$
whenever $i,j,h$ are distinct elements of $A$.
We first show that either $A=[m]$ or $|A|=2$.
Suppose that $A\neq[m]$ and choose $p\in[m]\setminus A$.
Assume, by contradiction, that $|A|\ge3$.
For any distinct $i,j,h\in A$, we have
$r_i+r_j+r_h\ge k$
and $r_h\le k/2$, which imply
$r_i+r_j\ge k/2.$
Consequently,
$\alpha_p+r_i+r_j>k$
for every pair of distinct indices $i,j\in A$.
Likewise,
$\alpha_p+\alpha_q+r_i\ge
\alpha_p+\alpha_q>
\frac{k}{2}+\frac{k}{2}=k$
for every $q\neq p$. Since $k-r_i\geq r_i$ for all $i$, replacing some $r_\ell$'s by $k-r_\ell$ in the inequalities $\alpha_p+r_i+r_j>k$ and $\alpha_p+\alpha_q+r_i>k$ still yield strict inequalities.
Thus for any $x_{p,\ell}$ with $p\in [m]\setminus A$, the monomial $u/x_{p,\ell}$ still belongs to $I^{(k)}$,
contradicting the minimality of $u$ given in \eqref{generator}. Hence $|A|\le2$.
On the other hand, if $|A|\le1$, the same argument shows that one may
again decrease one of the exponents of $u$ (corresponding to a variable $x_{p,\ell}$ with $p\in [m]\setminus A$),  without leaving
$I^{(k)}$, contradicting minimality. Thus $|A|=2$ whenever
$A\neq[m]$, proving the claim.

It follows that a monomial $u$ is a minimal generator of $I^{(k)}$ if
and only if one of the following holds.

\begin{enumerate}
	\item
	$u=\prod_{i=1}^{m}w_{i,j_i,r_i}$,
	where $j_i\in[d_i]$, $0\le r_i\le k/2$ for all $i$, $r_i+r_j+r_h\ge k$
	whenever $i,j,h$ are distinct, and for every $i\in [s]$ with $r_i>0$, there exist distinct $j,k\in [m]\setminus \{i\}$ such that $r_i+r_j+r_h=k$.
	
	\item
	$u=u_{i_1,j_1,r_1,i_2,j_2,r_2}
	=
	w_{i_1,j_1,r_1}
	w_{i_2,j_2,r_2}
	\prod_{i\neq i_1,i_2}
	\mathbf{x}_i^{\,k-(r_1+r_2)}$,
	where $i_1\neq i_2$, $j_t\in[d_{i_t}]$, and
	$r_1+r_2<k/2$.
\end{enumerate}

To every generator $u$ of type \textup{(1)} we associate the vector $f_u=(r_{s+1},\ldots,r_m,r_{q_1},\ldots,r_{q_s})$, where $\{q_1,\ldots,q_s\}=[s]$ and $r_{q_1}\leq \cdots \leq r_{q_s}$. Moreover, for a generator $v=u_{i_1,j_1,r_1,i_2,j_2,r_2}$ of type \textup{(2)} we define
$f_v=(\min\{r_1,r_2\},\max\{r_1,r_2\})$. 

If $s\geq 3$, then for any monomial $u=\prod_{i=1}^{m}w_{i,j_i,r_i}$ of type (1) with  $f_u=(r_{s+1},\ldots,r_m,r_{q_1},\ldots,r_{q_s})$, we have either $r_{q_2}=\cdots=r_{q_s}$, or $r_{q_2}<r_{q_{3}}=\cdots=r_{q_s}$, with $r_{q_{3}}=\cdots=r_{q_s}=k-(r_{q_1}+r_{q_2})$. 
To see this, suppose $r_{q_h}<r_{q_{h+1}}$, for some $2\leq h<s$. Since $u$ is minimal, $r_{q_{h+1}}+r_i+r_j=k$ for some $i$ and $j$ with $i,j$ and $q_{h+1}$ distinct. If  $q_2\notin\{i,j\}$, then  $r_{q_2}+r_i+r_j\geq k=r_{q_{h+1}}+r_i+r_j$, which yields $r_{q_2}\geq r_{q_{h+1}}$, a contradiction. So we have $q_2\in\{i,j\}$. Similarly, since $r_{q_1},\ldots,r_{q_h}<r_{q_{h+1}}$,  we have $q_1,\ldots,q_h\in\{i,j\}$. Hence, $h=2$ and
$\{i,j\}=\{q_1,q_2\}$. Moreover, $r_{q_{h+1}}=k-(r_{q_1}+r_{q_2})$. Replacing $h+1$ by any integer $h+1\leq \ell\leq s$, the same argument shows that $r_{q_{\ell}}=k-(r_{q_1}+r_{q_2})$.

We introduce a total order $>$ on $\mathcal G(I^{(k)})$ as follows.

\begin{enumerate}
	\item[(a)]
	Every generator of type \textup{(1)} is larger than every generator of
	type \textup{(2)}.
	
	\item[(b)]
	If two generators $u$ and $v$ are of the same type, we set $u>v$,  if $f_u>_{\lex} f_v$ or $f_u=f_v$ and $u>_{\lex}v$, where $>_{\lex}$ denotes the pure 
	lexicographic order
	induced by $x_{s+1,1}>\cdots> x_{s+1,d_{s+1}}>\cdots >x_{m,1}>\cdots >x_{m,d_m}>x_{1,1}>\cdots> x_{1,d_1}>\cdots >x_{s,1}>\cdots >x_{s,d_s}$.
\end{enumerate}

We claim that this order is admissible.	
Consider two minimal monomial generators $u$ and $v$ with $u>v$.

\textbf{Case 1}. 
Let $u=\prod_{i=1}^{m}w_{i,j_i,r_i}$
be of type \textup{(1)} and let $v=u_{i_1,h_1,s_1,i_2,h_2,s_2}$
be of type \textup{(2)}. We need to find  a variable $x_{a,b}|u:v$ and a monomial $z\in \mathcal G(I^{(k)})$ with $z>v$ such that $z:v=x_{a,b}$.
Since
$s_1+s_2<k/2$ and $r_{i_1}+r_{i_2}\ge k/2$,
at least one of the inequalities
$r_{i_1}>s_1$ and $r_{i_2}>s_2$
must hold. Without loss of generality, assume that
$r_{i_1}>s_1$.
The exponent of $x_{i_1,h_1}$ in $u$ is either
$r_{i_1}$ or $k-r_{i_1}$. Since
$k-r_{i_1}\ge r_{i_1}>s_1,$
it follows that
$x_{i_1,h_1}\mid u:v.$
If $s_1+s_2+1<k/2$, then $z=u_{i_1,h_1,s_1+1,i_2,h_2,s_2}\in \mathcal G(I^{(k)})$ with $f_z>_{\lex}f_v$ and $z:v=x_{i_1,h_1}$. Now, let
$s_1+s_2+1=k/2$. Then $k$ is even, and we take $z=w_{i_1,h_1,s_1+1}w_{i_2,h_2,s_2}\prod_{i\neq i_1,i_2}w_{i,1, k/2}$. Clearly, $z\in \mathcal G(I^{(k)})$, with $z>v$ and 
$z:v=x_{i_1,h_1}$. Finally, let $s_1+s_2=\lfloor k/2\rfloor$.
Then $k$ is necessarily odd. If $s_1\leq \lfloor k/2\rfloor-1$, then we choose $z=w_{i_1,h_1,s_1+1}w_{i_2,h_2,s_2}\prod_{i\neq i_1,i_2}w_{i,1, \lfloor k/2\rfloor}$. Since $s_1+s_2+1+\lfloor k/2\rfloor=k$, we conclude that  $z\in \mathcal G(I^{(k)})$. Moreover,  $z>v$ and 
$z:v=x_{i_1,h_1}$, so we are done. It remains to consider the case $s_1+s_2=\lfloor k/2\rfloor$ and $s_1=\lfloor k/2\rfloor$. Then clearly $s_2=0$. On the other hand, having $k$ odd, the inequalities $r_{i_1}+r_{i_2}\ge \lfloor k/2\rfloor+1$ and $r_{i_1}\le \lfloor k/2\rfloor$ imply that $r_{i_2}\geq 1$. Hence, 	$x_{i_2,h_2}\mid u:v$. So it is enough to consider $z=w_{i_1,h_1,s_1}w_{i_2,h_2,s_2+1}\prod_{i\neq i_1,i_2}w_{i,1, \lfloor k/2\rfloor}$. Then  $z\in \mathcal G(I^{(k)})$, $z>v$ and 
$z:v=x_{i_2,h_2}$.

\smallskip

\textbf{Case 2}. Suppose that $u$ and $v$ are of type \textup{(1)}. Let $u=\prod_{i=1}^{m}w_{i,j_i,r_i}$ and $v=\prod_{i=1}^{m}w_{i,h_i,t_i},$
and assume that $u>v$.
Suppose first that $f_u>_{\lex}f_v$, and let $s+1\leq p\leq m$ be the index for which
$r_p>t_p$ and $r_\ell=t_\ell$ for $\ell<p$. 
Then $x_{p,h_p}\mid u:v$ and 
$z=(v/w_{p,h_p,t_p})
w_{p,h_p,t_p+1}\in I^{(k)}$. Choose $z'\in \mathcal{G}(I^{(k)})$ such that $z'|z$. Notice that for any  $s+1\leq i\leq m$ and $j\in [d_i]$ we have $\deg_{z'}(x_{i,j})=\deg_{z}(x_{i,j})$. Thus $f_{z'}>_{\lex} f_v$. Moreover, $z'$
is a minimal generator of type \textup{(1)} satisfying
$z':v=z:v=x_{p,h_p}$.
Now suppose that $r_{\ell}=t_\ell$ for any $s+1\leq \ell\leq m$. Let $f_u=(r_{s+1},\ldots,r_m,r_{a_1},\ldots,r_{a_s})$ and $f_v=(t_{s+1},\ldots,t_m,t_{b_1},\ldots,t_{b_s})$. Since $f_u>_{\lex} f_v$, there exists $p$ such that  $r_{a_{\ell}}=t_{b_{\ell}}$ for any $\ell<p$ and $r_{a_p}>t_{b_p}$. First assume that $p=1$ and $r_{a_1}>t_{b_1}$. Then $r_{b_1}\geq r_{a_1}>t_{b_1}$, which implies that $x_{b_1,1}| u:v$. Since $u$ is minimal, there exist  $j_1,j_2\in [m]$ such that $j_1$, $j_2$ and $b_1$ are distinct and $r_{b_1}+r_{j_1}+r_{j_2}=k$. We have $t_{b_1}+t_{j_1}+t_{j_2}\geq k=r_{b_1}+r_{j_1}+r_{j_2}$. Since $r_{b_1}>t_{b_1}$, we have $t_{j_1}>r_{j_1}$ or $t_{j_2}>r_{j_2}$. We may assume that $t_{j_1}>r_{j_1}$. Then  $j_1\in [s]$ and hence $j_1=b_d$ for some $d\in [s]$. So $t_{b_d}>r_{b_d}\geq r_{a_1}>t_{b_1}$ and hence, $t_{b_d}\geq t_{b_1}+2$.
We show that if $t_{b_d}+t_{\alpha}+t_{\beta}=k$, where $\alpha$, $\beta$ and $b_d$ are distinct, then $b_1\in \{\alpha,\beta\}$. Suppose on the contrary that  $b_1\notin \{\alpha,\beta\}$. Then $t_{b_1}+t_{\alpha}+t_{\beta}\geq k=t_{b_d}+t_{\alpha}+t_{\beta}$, which yields $t_{b_1}\geq t_{b_d}$, a contradiction. Hence,  $b_1\in \{\alpha,\beta\}$. Similarly, $\{b_\ell\in [s]: t_{b_\ell}<t_{b_d}\}\subseteq \{\alpha,\beta\}$.  This shows that $z=(x_{b_1,1}v)/x_{b_d,1}\in  I^{(k)}$. Let $z'\in\mathcal{G}(I^{(k)})$
such that $z'|z$. Then it follows from the discussion above that $\deg_{z'}(x_{b_1,1})=t_{b_1}+1$, $\deg_{z'}(x_{b_d,1})=t_{b_d}-1$, and for any $\ell\in[s]\setminus\{1\}$ with  $t_{b_\ell}<t_{b_d}$, $\deg_{z'}(x_{b_\ell,1})=\deg_{v}(x_{b_\ell,1})=t_{b_\ell}$. Moreover, for any 
$\ell\in[s]$ with  $t_{b_\ell}\geq t_{b_d}$, we have $\deg_{z'}(x_{b_\ell,1})\geq \deg_{v}(x_{b_\ell,1})-1=t_{b_\ell}-1\geq  t_{b_1}+1$. 
Thus $f_{z'}>_{\lex} f_v$ and
$z':v=z:v=x_{b_1,1}$. 

Now, suppose $r_{a_1}=t_{b_1}$ and $r_{a_2}>t_{b_2}$. If $a_1\neq b_1$, then $r_{b_1}\geq r_{a_2}>t_{b_2}\geq t_{b_1}$. Hence, as before $x_{b_1,1}| u:v$, and
there exists $d\in[s]$ such that $t_{b_d}>r_{b_d}\geq r_{a_1}=t_{b_1}$. If $t_{b_d}\geq t_{b_1}+2$ for some $d\in[s]$, then we take $z=(x_{b_1,1}v)/x_{b_d,1}$ and $z'\in\mathcal{G}(I^{(k)})$
such that $z'|z$. Then the same argument as above yields the desired result.
Now, let $t_{b_\ell}\leq t_{b_1}+1$
for all $\ell\in[s]$. We show that this case cannot occur by deriving a contradiction.
We have $r_{a_2}>t_{b_1}=r_{a_1}$. Hence, $r_{a_1}+r_{a_2}+r_h=k$ for some $h\neq a_1,a_2$. 
If $s+1\leq h\leq m$, then  $t_h=r_h$ and hence
$t_{b_1}+t_{b_2}+t_h<r_{a_1}+r_{a_2}+r_h=k$, a contradiction. Hence, $h\in[s]$ and $h=a_j$ for some $3\leq j\leq s$. Thus $r_{a_1}+r_{a_2}+r_{a_j}=k$ and $r_{a_j}\geq r_{a_2}\geq t_{b_1}+1\geq t_{b_\ell}$ for any $\ell\in [s]$.
So  $t_{b_1}+t_{b_2}+t_{b_3}<r_{a_1}+r_{a_2}+r_{a_j}=k$, a contradiction. 

So we assume  $r_{a_1}=t_{b_1}$,  $r_{a_2}>t_{b_2}$  and $a_1=b_1$. Notice that $b_2\neq a_1$ and so $r_{b_2}\geq r_{a_2}>0$. 
Since $u$ is minimal, there exist  $j_1$ and $j_2$ such that $j_1$, $j_2$ and $b_2$ are distinct and $r_{b_2}+r_{j_1}+r_{j_2}=k$.  We have $t_{b_2}+t_{j_1}+t_{j_2}\geq k=r_{b_2}+r_{j_1}+r_{j_2}$. Since $a_1=b_1$, we have $r_{b_2}\geq r_{a_2}>t_{b_2}$. So we conclude that $t_{j_1}>r_{j_1}$ or $t_{j_2}>r_{j_2}$. We may assume that $t_{j_1}>r_{j_1}$. Then $j_1=b_d$ for some $d\in [s]\setminus\{1\}$.  
Notice that $t_{b_d}>r_{b_d}\geq r_{a_2}>t_{b_2}\geq t_{b_1}$ and so $t_{b_d}\geq t_{b_2}+2$. Moreover, it follows from $t_{b_d}>t_{b_2}\geq t_{b_1}$ that $t_{b_\ell}=t_{b_d}=k-(t_{b_1}+t_{b_2})$ for any $3\leq\ell\leq s$.
This shows that $z=(x_{b_2,1}v)/x_{b_d,1}\in  I^{(k)}$, and if  $z'$
is a minimal generator satisfying $z'|z$, then
$\deg_{z'}(x_{b_1,1})=\deg_{v}(x_{b_1,1})=t_{b_1}$,
$\deg_{z'}(x_{b_2,1})=\deg_{v}(x_{b_2,1})+1=t_{b_2}+1$,
and $\deg_{z'}(x_{b_\ell,1})\geq \deg_{v}(x_{b_\ell,1})-1=t_{b_\ell}-1=t_{b_d}-1\geq t_{b_2}+1$ for any $3\leq\ell\leq s$. Thus
$f_{z'}>_{\lex} f_v$ and
$z':v=z:v=x_{b_2,1}$. Since $r_{b_2}>t_{b_2}$, we have $x_{b_2,1}| u:v$, and we are done.
  
Next, suppose that     
$r_{a_1}=t_{b_1}$ and $r_{a_2}=t_{b_2}$. We show that $r_{a_\ell}=t_{b_\ell}$ for any $1\leq \ell\leq s$. Suppose on the contrary that $r_{a_\ell}>t_{b_\ell}$ for some $3\leq \ell\leq s$. Then $r_{a_\ell}>t_{b_2}=r_{a_2}$, and so $r_{a_\ell}=k-(r_{a_1}+r_{a_2})=k-(t_{b_1}+t_{b_2})\leq t_{b_\ell}$, a contradiction.
Therefore, $r_{a_\ell}=t_{b_\ell}$ for all $\ell$. So $f_u=f_v$, which implies that
$u>_{\lex}v$.
Let $x_{p,\ell}$ denote the largest variable dividing $u:v$. First assume that $s+1\leq p\leq m$. 
Since
$r_p=s_p$ and $u>_{\lex}v$,
necessarily
$\ell=h_p$ and $h_p<j_p$.
Assume first that
$s_p\le \lfloor k/2\rfloor-1.$
Then
$z=(v/{w_{p,h_p,s_p}})w_{p,h_p,s_p+1}\in  I^{(k)}$.  Choose $z'\in \mathcal{G}(I^{(k)})$ such that $z'|z$. Notice that for any  $s+1\leq i\leq m$ and $j\in [d_i]$ we have $\deg_{z'}(x_{i,j})=\deg_{z}(x_{i,j})$. Thus $f_{z'}>_{\lex}f_v$. Moreover,  
$z':v=z:v=x_{p,h_p}$.
It remains to consider the case
$s_p=\lfloor k/2\rfloor.$
Then $k$ must be odd, otherwise
$\deg_u(x_{p,h_p})=\deg_v(x_{p,h_p}),$
contrary to the choice of $x_{p,h_p}$.
So $k$ is odd and $k-s_p=k-\lfloor k/2\rfloor=\lfloor k/2\rfloor+1$. Take $z=(x_{p,h_p}v)/x_{p,j_p}=(v/w_{p,h_p,s_{p}})w_{p,j_p,s_{p}}$. Notice that $z\in \mathcal{G}(I^{(k)})$. Moreover, from $h_p<j_p$ we obtain $z>_{\lex}v$. Since $f_z=f_v$, we conclude that $z>v$ with $z:v=x_{p,h_p}$.   
Now, suppose that $1\leq p\leq s$, which means that $x_{p,1}|u:v$, $r_p>t_p$, and $r_{\ell}=t_{\ell}$ for $\ell<p$.  Since $\{r_1,\ldots,r_s\}=\{t_1,\ldots,t_s\}$, there exists $q>p$ such that $r_q<t_q$. We show that $t_p<t_q$.  We know that $|\{t_1,\ldots,t_s\}|\leq 3$. 
Since $r_p,r_q\in \{t_1,\ldots,t_s\}$, we have $r_p=t_i$ and $r_q=t_j$ for some $i$ and $j$. By contradiction, suppose that $t_p\geq t_q$. Then $t_i>t_p\geq t_q>t_j$. Since $p$, $q$, $j$ and $i$ are pairwise distinct, we must have $t_i=t_p$, a contradiction. Therefore, $t_p<t_q$. So for any triple $(t_q,t_h,t_\ell)$ satisfying $t_q+t_h+t_\ell=k$ with $q$, $h$ and $\ell$ distinct, we have $p\in\{h,\ell\}$, otherwise, $t_p+t_h+t_\ell<k$ with $p$, $h$ and $\ell$ distinct, which is a contradiction.
This yields that $t_q+t_p+t_d=k$ for some $d\neq p, q$, and that $z=(x_{p,1}v)/x_{q,1}\in I^{(k)}$. Let $z'\in \mathcal{G}(I^{(k)})$ such that $z'|z$, and let $f_{z'}=(t'_{s+1},\ldots,t'_m,t'_{c_1},\ldots,t'_{c_s})$.
We show that either $f_{z'}>_{\lex} f_v$ or $f_{z'}=f_v$ and $z'>_{\lex} v$. 
For any $s+1\leq i\leq m$, there exist $j,h\in [d_i]$ such that $\deg_z(x_{i,j})=t_i$ and $\deg_z(x_{i,h})=k-t_i$. Hence, $t'_i=t_i$. Moreover,  the equality $(t_q-1)+(t_p+1)+t_d=k$ implies that $t'_q=t_q-1$, $t'_p=t_p+1$ and $t'_d=t_d$. If for any  $i\in [s]\setminus \{q,p\}$, $t'_i=t_i$, then $z'=z$ and we have either $f_{z'}>_{\lex} f_v$ or $f_{z'}=f_v$ and $z'>_{\lex} v$. Now,
assume that there exists $i\in [s]\setminus \{q,p\}$ such that $t'_i\neq t_i$. Then $t'_i<t_i$. We show that $t_i=t_q$ and $t_i>t_p+1$. 
There are distinct integers $h,\ell\neq i$ such that $t_i+t_h+t_\ell=k$. If $p\notin\{h,\ell\}$, then $t'_h\leq t_h$, and $t'_\ell\leq t_\ell$, which yields $t'_i+t'_h+t'_\ell<k$, a contradiction. So
we have $p\in\{h,\ell\}$, and we may write $t_i+t_p+t_\ell=k$. If $\ell=q$, then $t'_i+t'_p+t'_q<t_i+t_p+t_q=k$, a contradiction. Thus $\ell\neq q$, and we obtain $t_q+t_p+t_\ell\geq k=t_i+t_p+t_\ell$. This gives $t_q\geq t_i$. Moreover,  
$t_q+t_p+t_d=k\leq t_i+t_p+t_d$, as $d\neq i$. So $t_q\leq t_i$. Therefore, $t_i=t_q$.
Moreover,  since $t_i+t_p+t_\ell=k$, $t'_p=t_p+1$ and $t'_\ell\leq t_\ell$, we must have $t'_i=t_i-1$, otherwise, $t'_i+t'_p+t'_\ell<k$.
Hence, $t'_i=t_i-1$. 
Now, we show that  $t_i>t_p+1$. Suppose, on the contrary, that $t_i\leq t_p+1$. Then $t_q=t_i\leq t_p+1$ and since $t_q>t_p$ we obtain $t_q=t_i=t_p+1$. Then 
$t'_i+t'_q+t'_d=t_p+(t_q-1)+t_d=k-1<k$, a contradiction. We have shown that $t_i>t_p+1$, for any $i\in [s]\setminus \{q,p\}$ with $t'_i<t_i$.  
It follows that $t'_j=t_j$, for any $j\in [s]\setminus \{q,p\}$ with $t_j\leq t_p+1$.
This together with $t'_p=t_p+1$ implies $f_{z'}>_{\lex} f_v$. Note that $z':v=z:v=x_{p,1}$.

\smallskip

\textbf{Case 3}. Suppose
$u=u_{i_1,j_1,r_1,i_2,j_2,r_2}$ and $v=u_{t_1,h_1,s_1,t_2,h_2,s_2}$ are of type \textup{(2)}. 
Without loss of generality, assume
$r_1\le r_2$ and
$s_1\le s_2.$
Then either
$(r_1,r_2)>_{\lex}(s_1,s_2),$
or
$(r_1,r_2)=(s_1,s_2)$ and $u>_{\lex}v.$
Suppose first that
$r_1>s_1.$
Since
$k-r_1\ge k-r_2\ge r_2\ge r_1>s_1$
and
$k-(r_1+r_2)>k/2>s_1,$
it follows that
$x_{t_1,h_1}\mid u:v.$
If $s_1+s_2+1<k/2$, then $z=u_{t_1,h_1,s_1+1,t_2,h_2,s_2}\in \mathcal G(I^{(k)})$ with $f_z>_{\lex}f_v$ and $z:v=x_{t_1,h_1}$. If  $s_1+s_2+1=k/2=\lfloor k/2\rfloor$, then $k$ is even and we set $z=w_{t_1,h_1,s_1+1}w_{t_2,h_2,s_2}\prod_{i\neq t_1,t_2}w_{i,1, k/2}$. Clearly, $z\in \mathcal G(I^{(k)})$, with $z>v$ and $z:v=x_{t_1,h_1}$.
Now, assume that $s_1+s_2=\lfloor k/2\rfloor$. Since $s_1+s_2<k/2$, we conclude that $k$ is odd. 
Moreover, from $s_1\leq s_2$, we obtain $s_1\leq \lfloor k/2\rfloor-1$. Then $z=w_{t_1,h_1,s_1+1}w_{t_2,h_2,s_2}\prod_{i\neq t_1,t_2}w_{i,1, \lfloor k/2\rfloor}\in \mathcal G(I^{(k)})$ is  the desired monomial. 
Next, suppose that
$r_1=s_1$ and $r_2>s_2.$ Then $s_1+s_2<r_1+r_2$, which implies that $s_1+s_2\leq \lfloor k/2\rfloor-1$.  
If
$x_{i_1,j_1}\neq x_{t_1,h_1},$
then
$x_{t_1,h_1}\mid u:v$, and we take	\[
z=
\begin{cases}
	u_{t_1,h_1,s_1+1,t_2,h_2,s_2}, & \text{if $s_1+s_2+1<k/2$}, \\
	w_{t_1,h_1,s_1+1}w_{t_2,h_2,s_2}\prod_{i\neq t_1,t_2}w_{i,1, k/2}, & \text{if $s_1+s_2+1=\lfloor k/2\rfloor=k/2$}.
\end{cases}
\]
Then $z\in \mathcal G(I^{(k)})$ is  the desired monomial. 
Now, let
$x_{i_1,j_1}=x_{t_1,h_1}.$ Since $k-r_2\geq r_2>s_2$, we have $x_{t_2,h_2}\mid u:v$. Moreover, $s_2\leq \lfloor k/2\rfloor-1$. We set
\[
z=
\begin{cases}
	u_{t_1,h_1,s_1,t_2,h_2,s_2+1}, & \text{if $s_1+s_2+1<k/2$}, \\
	w_{t_1,h_1,s_1}w_{t_2,h_2,s_2+1}\prod_{i\neq t_1,t_2}w_{i,1, k/2}, & \text{if $s_1+s_2+1=\lfloor k/2\rfloor=k/2$}.
\end{cases}
\]
As before, $z\in \mathcal G(I^{(k)})$ is  the desired monomial.

It remains to consider the case $(r_1,r_2)=(s_1,s_2)$, so that $u>_{\mathrm{lex}}v$.
We first show that either $x_{t_1,h_1}\mid u:v$ or $x_{t_2,h_2}\mid u:v$. If $s_1=s_2$, then
$\{x_{i_1,j_1},x_{i_2,j_2}\}\neq\{x_{t_1,h_1},x_{t_2,h_2}\}$, otherwise $u=v$. Hence, $\deg_u(x_{t_1,h_1})>s_1$ or $\deg_u(x_{t_2,h_2})>s_2$, which proves the claim. 
Assume now that $s_1<s_2$. If $x_{i_1,j_1}\neq x_{t_1,h_1}$, then $\deg_u(x_{t_1,h_1})>s_1$, and therefore $x_{t_1,h_1}\mid u:v$. Otherwise, $x_{i_1,j_1}=x_{t_1,h_1}$, forcing $x_{i_2,j_2}\neq x_{t_2,h_2}$, whence $x_{t_2,h_2}\mid u:v$. This establishes the claim.
Choose $\ell\in\{1,2\}$ to be the smallest integer such that $x_{t_\ell,h_\ell}\mid u:v$, and let $\{1,2\}\setminus \{\ell\}=\{q\}$. 
Suppose first that $s_1+s_2\le\lfloor k/2\rfloor-1$.   
We set 	\[
z=
\begin{cases}
	u_{t_\ell,h_\ell,s_\ell+1,t_q,h_q,s_q}, & \text{if $s_1+s_2+1<k/2$}, \\
	w_{t_\ell,h_\ell,s_\ell+1}w_{t_q,h_q,s_q}\prod_{i\neq t_1,t_2}w_{i,1, k/2}, & \text{if $s_1+s_2+1=k/2$}.
\end{cases}
\]

Finally, assume that $s_1+s_2=\lfloor k/2\rfloor$. Since $s_1+s_2<k/2$,  $k$ is odd. Moreover, from $s_1\leq s_2$, we obtain $s_1\leq \lfloor k/2\rfloor-1$ and $s_2>0$.
Set 
\[
z=
\begin{cases}
	w_{t_1,h_1,s_1+1}w_{t_2,h_2,s_2}\prod_{i\neq t_1,t_2}w_{i,1, \lfloor k/2\rfloor}, & \text{if $\ell=1$},\\
	w_{t_1,h_1,s_1}w_{t_2,h_2,s_2+1}\prod_{i\neq t_1,t_2}w_{i,1, \lfloor k/2\rfloor}, & \text{if $\ell=2$ and  $s_2\leq \lfloor k/2\rfloor-1$},\\
	(x_{t_2,h_2}v)/x_{i_2,j_2}, & \text{if $\ell=2$ and  $s_2=\lfloor k/2\rfloor$}. 
\end{cases}
\]

Then $z\in \mathcal G(I^{(k)})$ is  the desired monomial. Indeed, if  $\ell=2$ and  $s_2=\lfloor k/2\rfloor$, then $s_1=0$, and by the choice of $\ell$ we obtain $x_{i_1,j_1}=x_{t_1,h_1}$.
In this case, $\deg_u(x_{a,b})=\deg_v(x_{a,b})$ for every variable $x_{a,b}\notin\{x_{i_2,j_2},x_{t_2,h_2}\}$, whereas
$\deg_v(x_{i_2,j_2})=k-(s_1+s_2)=k-s_2=\lfloor k/2\rfloor+1>\lfloor k/2\rfloor=\deg_u(x_{i_2,j_2})$ and $\deg_u(x_{t_2,h_2})=\lfloor k/2\rfloor+1>\lfloor k/2\rfloor=\deg_v(x_{t_2,h_2})$. Since $u>_{\mathrm{lex}}v$, we necessarily have $x_{t_2,h_2}>x_{i_2,j_2}$.
So defining $z=(x_{t_2,h_2}v)/x_{i_2,j_2}$, we have
\[
z=
\begin{cases}
	u_{t_1,h_1,s_1,t_2,j_2,s_2}, & \text{if $t_2=i_2$},\\
	
	u_{t_1,h_1,s_1,i_2,j_2,s_2}, & \text{if $t_2\neq i_2$}. 
\end{cases}
\]

Hence, $z\in \mathcal G(I^{(k)})$  with  $z:v=x_{t_2,h_2}$, while $f_z=f_v$ and $z>_{\mathrm{lex}}v$. Hence $z>v$, completing the proof.
\end{proof}

Thanks to Theorem \ref{LQ of complete multi}, we determine the sequences of depth of symbolic powers of $I_{c}(K_{d_1,\dots,d_m})$ under the assumption that $d_{i}\geq2$ for all $i$. In particular, these sequences show that the depths of symbolic powers of complementary edge ideals are not necessarily non-increasing.  

\begin{prop}\label{depth of complete multi}
Let $G=K_{d_1,\dots,d_m}$ be a complete multipartite graph on $n\geq 3$ vertices. If $d_{i}\geq2$ for all $i$, then we have 
\[
\depth S/I_{c}(G)^{(k)}=
\begin{cases}
n-3, & \text{if } k=1,2,4, \\
n-m-1, & \text{otherwise}. 
\end{cases}
\]
In particular, if $m\geq 3$, then we have $$\depth S/I_{c}(G)^{(3)}=n-m-1<n-3=\depth S/I_{c}(G)^{(4)}.$$
\end{prop}
\begin{proof}
Set $I=I_{c}(G)$. First, we suppose that $m=2$. As shown in the proof of Theorem \ref{LQ of complete multi}, we have $I^{(k)}=J(K_{d_{1}})^{(k)}J(K_{d_{2}})^{(k)}$, where $K_{d_{i}}$ denotes the complete graph on the vertex set $X_{i}=\{x_{i,\ell}\,\,:\ell\in[d_{i}]\}$ for each $i$. Notice that the ideal $S/J(K_{d_{i}})^{(k)}$ is Cohen--Macaulay of dimension $n-2$ (see \cite[Theorem 2.1]{v2011}). Hence, 
for each $i=1,2$, $\pd J(K_{d_{i}})^{(k)}=1$, and thus, $\pd I^{(k)}=2$, which implies that $\depth S/I^{(k)}=n-3=n-m-1$. From now on, we assume that $m\geq3$. We keep the notation as in the proof of Theorem~\ref{LQ of complete multi}, and let $>$ be the admissible order given there. 
Let $v=\prod_{1\leq i\leq m}w_{i,h_{i},t_{i}}$ be a generator of type (1). The proof of Case 2 in Theorem \ref{LQ of complete multi} shows that only the indices of the variables $x_{i,h_{i}}$ with $t_{i}<k/2$ can belong to $\set_{I^{(k)}}(v)$. Therefore, we obtain $$|\set_{I^{(k)}}(v)|\leq|\{i\,\,: t_{i}<k/2\}|\leq m.$$ Next, let $v=u_{i_{1},h_{1},t_{1},i_{2},h_{2},t_{2}}$ be generator of type (2). Cases 1 and 3 show that only the indices of $x_{i_{1},h_{1}}$ and $x_{i_{2},h_{2}}$ can belong to $\set_{I^{(k)}}(v)$. Hence, we obtain $|\set_{I^{(k)}}(v)|\leq2$. Thus by \cite[Corollary 8.2.2]{HHBook} and the Auslander--Buchsbaum formula, $\depth S/I^{(k)}\geq n-m-1$.  Now, suppose that $k\neq 1,2,4$. We need to show that $\depth S/I^{(k)}\leq n-m-1$ or equivalently, $\pd I^{(k)}\geq m$.  Let $u=\prod_{i=1}^m w_{i,1,r_i}$ be a minimal generator of type (1) defined by $r_i=\lceil k/2\rceil-1$ for any $i\in[m]$. Since $k\neq 1,2,4$, it is seen that $u$ is indeed a minimal generator of $I^{(k)}$. Fix an integer $\ell\in[m]$. If $k$ is odd, then we set $u_\ell=w_{\ell,2,r_\ell}\prod_{i=1,i\neq \ell}^m w_{i,1,r_i}$ and if $k$ is even we set $u_\ell=w_{\ell,1,k/2}\prod_{i=1,i\neq \ell}^m w_{i,1,r_i}$. It is easily seen that $u_\ell>u$. Moreover, $u_\ell:u=x_{\ell,1}$ for all $\ell\in[m]$. Therefore, $\{x_{1,1},\ldots,x_{m,1}\}\subseteq \set_{I^{(k)}}(u)$, and so $\pd I^{(k)}\geq m$, which completes the proof for the case $k\neq 1,2,4$. 

We now consider $k=1,2,4$. 
For $k=1$, every minimal generator of $I$ is of the form $v={\bf x}/(x_{i,a}x_{j,b})$ with $i\neq j$, where ${\bf x}={\bf x}_{1}\cdots{\bf x}_{m}$. Since $I$ is squarefree, it is easily seen that $\set_{I}(v)\subseteq \{x_{i,a},x_{j,b}\}$ and so $|\set_{I}(v)|\leq 2$. Moreover, since $d_i\geq 2$ for all $i$,  the monomials $v={\bf x}/(x_{1,1}x_{2,1})$,  $u=x_{1,1}(v/x_{1,2})$ and $w=x_{2,1}(v/x_{2,2})$ are minimal generators of $I$ with $u:v=x_{1,1}$ and $w:v=x_{2,1}$ and $u,w>v$. Hence, $|\set_{I}(v)|=2$. This proves the claim for $k=1$.
Let $k=2$. For a generator of type (1), each $t_{i}$ is either 0 or 1, and $t_{i}+t_{j}+t_{h}\geq2$ for any pairwise distinct indices $i,j,h$. Thus, $t_{i}=0$ for at most one index $i$. By the fact that $|\set_{I^{(k)}(v)}|\leq|\{i\,\,: t_{i}<k/2\}|$, we have $|\set_{I^{(2)}}(v)|\leq1$ for a generator $v$ of type (1). Also, since $|\set_{I^{(k)}}(v)|\leq2$, we see that $\max_{v\in\mathcal{G}(I^{(2)})}|\set_{I^{(2)}}(v)|\leq2$. Let $k=4$. For a generator of type (1), each $t_{i}\in\{0,1,2\}$. We claim that there are at most two indices $i$ such that $t_{i}<2$. Indeed, suppose that there exist pairwise distinct indices $i,j,h$ such that $t_{i}, t_{j}, t_{h}<2$. Since $t_{i}, t_{j}, t_{h}\in\{0,1\}$, we have $t_{i}+t_{j}+t_{h}\leq3$, a contradiction. Hence, we obtain $\max_{v\in\mathcal{G}(I^{(4)})}|\set_{I^{(4)}}(v)|\leq2$. Consider the prime ideal $P=(x_{1,1},x_{2,1},x_{3,1})$. Then we see that $P\in\Min(I)=\Min(I^{(k)})$, which implies that $\depth S/I^{(k)}\leq n-3$. From $\max_{v\in\mathcal{G}(I^{(k)})}|\set_{I^{(k)}}(v)|=2$ for $k=1,2,4$, we get $\depth S/I^{(k)}=n-3$ for $k=1,2,4$, which completes the proof. 
\end{proof}

We conclude the paper with the following description of the Waldschmidt constant.

\begin{prop}
\label{Waldschmit-constant}
Let $G=K_{d_1,\dots,d_m}$ be a complete multipartite graph on $n\geq 3$ vertices, and let $I=I_c(G)$. Set $s=|\{i\,\,: d_{i}=1\}|$.  Then 
 $$
        \widehat{\alpha}(I)=\begin{cases}
            \phantom{(} n/2,&\textit{if}\ s=0,\\
            \phantom{(}  n/3,&\textit{if}\ s=m,\\
            (n-1)/2,&\textit{if}\  s<m,\, 1\leq s\leq 3,\\
            (3n-s)/6,&\textit{if}\ 4\leq s<m.
        \end{cases}
        $$
\end{prop}

\begin{proof}
As was shown in the proof of Theorem~\ref{LQ of complete multi}, any minimal monomial generator of $I^{(k)}$ has one of the following forms:
\begin{enumerate}
	\item
	$u=\prod_{i=1}^{m}w_{i,j_i,r_i}$,
	where $j_i\in[d_i]$, $0\le r_i\le k/2$ for all $i$, $r_i+r_j+r_h\ge k$
	whenever $i,j,h$ are distinct, and for every $i\in [s]$ with $r_i>0$, there exist distinct $j,k\in [m]\setminus \{i\}$ such that $r_i+r_j+r_h=k$.
	
	\item
	$u=u_{i_1,j_1,r_1,i_2,j_2,r_2}
	=
	w_{i_1,j_1,r_1}
	w_{i_2,j_2,r_2}
	\prod_{i\neq i_1,i_2}
	\mathbf{x}_i^{\,k-(r_1+r_2)}$,
	where $i_1\neq i_2$, $j_t\in[d_{i_t}]$, and
	$r_1+r_2<k/2$.
\end{enumerate}

For any $0\le r_i\le k/2$, we set $p_i=r_i+(d_i-1)(k-r_i)=\deg(w_{i,j_i,r_i})$.  

\smallskip

\textbf{Case 1}. Suppose $s=0$. Then $d_i\geq 2$ for all $i$, and $p_i=k+(d_i-2)(k-r_i)$ is minimal precisely when $r_i=\lfloor k/2\rfloor$. So
for every generator $u$ of the form (1), the value $\deg(u)=\sum_{i=1}^m p_i$ is minimal precisely when  $r_i=\lfloor k/2\rfloor$ for all $i$. In this case $\deg(u)=m\lfloor k/2\rfloor+(n-m)\lceil k/2\rceil$. Hence,  $n\lfloor k/2\rfloor\leq \deg(u)\leq n\lceil k/2\rceil$. Now, let $u$ be of type (2). Then we have
\begin{equation*}
\begin{aligned}
\deg(u) &\geq p_1+p_2+(k-(r_1+r_2))+\lceil k/2\rceil\Bigl(\sum_{i=3}^{m} d_i-1\Bigr) \\
&\geq (d_1+d_2-4)\lceil k/2\rceil+3k-(r_1+r_2)+\Bigl(\sum_{i=3}^{m} d_i-1\Bigr)\lceil k/2\rceil \\
&\geq (n-5)\lceil k/2\rceil+2k+\lceil k/2\rceil \geq n\lfloor k/2\rfloor.
\end{aligned}
\end{equation*}
This shows that
$n\lfloor k/2\rfloor\leq \alpha(I^{(k)})\leq n\lceil k/2\rceil$, and so 
$\widehat{\alpha}(I)=n/2$.

\smallskip

\textbf{Case 2}. Suppose that $s=m$ that is $d_i=1$ for all $i$. Then $G$ is the complete graph on $[n]$. Consider a minimal $k$-cover ${\bf a}=(a_1,\ldots,a_n)$ of $I$. Writing the exponents of ${\bf a}$ in an increasing order we may assume that $a_{i_1}\leq \cdots\leq a_{i_n}$. The minimality of ${\bf a}$ implies that $a_{i_1}+a_{i_2}+a_{i_3}=k$ and $a_{i_j}=a_{i_3}$ for any $4\leq j\leq n$. Hence, $a_{i_j}=a_{i_3}\geq k/3$, and $u={\bf x}^{\bf a}$ has the minimal degree, precisely when $a_{i_j}=\lceil k/3\rceil$ for any $j\geq 3$. In this case $\deg(u)=k+(n-3)\lceil k/3\rceil$.  Hence, $\alpha(I^{(k)})=k+(n-3)\lceil k/3\rceil$ and $\widehat{\alpha}(I)=n/3$. 
\smallskip

\textbf{Case 3}. Suppose that  $s<m$ and $1\leq s\leq 3$.
We show that $$(n-1)\lfloor k/2\rfloor\leq \alpha(I^{(k)})\leq (n-1)\lceil k/2\rceil,$$ which proves the claim.
Let $d_1=\cdots=d_s=1$ and $d_{s+1},\ldots,d_m\geq 2$. For any generator $u$ of type (2), since $k-(r_1+r_2), k-r_1,k-r_2\geq \lceil k/2\rceil$, we have $\deg(u)\geq (n-3)\lceil k/2\rceil+k\geq (n-1)\lfloor k/2\rfloor$. Now, let $u$ be of type (1). First assume that $s=1$. Then $\deg(u)=r_1+\sum_{i=2}^m p_i\geq (m-1)\lfloor k/2\rfloor+(n-m)\lceil k/2\rceil\geq (n-1)\lfloor k/2\rfloor$. So $\alpha(I^{(k)})\geq (n-1)\lfloor k/2\rfloor$. On the other hand, setting $r_i=\lfloor k/2\rfloor$ for any $2\leq i\leq m$ and $r_1=k-2\lfloor k/2\rfloor$, we obtain a minimal generator $u$ of type (1) of degree $(m-3)\lfloor k/2\rfloor+(n-m)\lceil k/2\rceil+k$, which shows that $\alpha(I^{(k)})\leq (n-1)\lceil k/2\rceil$, as desired. Next, assume that $s=2$ and $u$ is of type (1). We have $r_1+r_2\geq k-r_3\geq \lceil k/2\rceil$. Hence, $\deg(u)= (r_1+r_2)+\sum_{i=3}^m p_i\geq \lceil k/2\rceil+(m-2)\lfloor k/2\rfloor+(n-m)\lceil k/2\rceil\geq (n-1)\lfloor k/2\rfloor$. So $\alpha(I^{(k)})\geq (n-1)\lfloor k/2\rfloor$. Now, choosing $r_1$ and $r_2$ such that $r_1+r_2=\lceil k/2\rceil$ and $r_1,r_2\geq k-2\lfloor k/2\rfloor$, and setting $r_i=\lfloor k/2\rfloor$ for any $3\leq i\leq m$, we obtain a minimal generator $u$ with $\deg(u)=\lceil k/2\rceil+(m-2)\lfloor k/2\rfloor+(n-m)\lceil k/2\rceil$. Therefore, $\alpha(I^{(k)})\leq (n-1)\lceil k/2\rceil$, and we are done.
Finally, assume that $s=3$ and $u$ is of type (1). We have $r_1+r_2+r_3\geq k$. Hence, $\deg(u)=(r_1+r_2+r_3)+\sum_{i=4}^m p_i\geq k+(m-3)\lfloor k/2\rfloor+(n-m)\lceil k/2\rceil\geq (n-1)\lfloor k/2\rfloor$. Thus, $\alpha(I^{(k)})\geq (n-1)\lfloor k/2\rfloor$. Now, choose $r_1,r_2$ and $r_3$ so that $r_1+r_2+r_3=k$ and $r_1,r_2,r_3\geq k-2\lfloor k/2\rfloor$, and set $r_i=\lfloor k/2\rfloor$ for any $4\leq i\leq m$. Then $u$ is a minimal generator with $\deg(u)=k+(m-3)\lfloor k/2\rfloor+(n-m)\leq (n-1) \lceil k/2\rceil$, which yields $\alpha(I^{(k)})\leq (n-1) \lceil k/2\rceil$.  

\smallskip

\textbf{Case 4}. Suppose $4\leq s<m$ with  $d_1=\cdots=d_s=1$ and $d_{s+1},\ldots,d_m\geq 2$.
We show that $$(n-s/3)\lfloor k/2\rfloor\leq \alpha(I^{(k)})\leq (3n-s)\lceil k/6\rceil.$$ This then yields the desired equality.
For any generator $u$ of type (2), we have $\deg(u)\geq (n-3)\lceil k/2\rceil+k\geq (n-1)\lfloor k/2\rfloor\geq  (n-s/3)\lfloor k/2\rfloor$. Now, let $u$ be of type (1). We may assume that $r_1\leq\cdots\leq r_m$. As was shown in the proof of Theorem~\ref{LQ of complete multi}, we have $r_i=r_3$ for and $4\leq i\leq s$. Since $r_1+r_2+r_3\geq k$, and $r_1,2_2\leq r_3$, we have $r_i=r_3\geq k/3$ for any $4\leq i\leq s$. This shows that 
$\deg(u)\geq k+(s-3)(k/3)+(m-s)\lfloor k/2\rfloor+(n-m)\lceil k/2\rceil\geq (s/3+n-s)\geq (2s/3)\lfloor k/2\rfloor+(n-s)\lfloor k/2\rfloor=(n-s/3)\lfloor k/2\rfloor$. This yields that $\alpha(I^{(k)})\geq (n-s/3)\lfloor k/2\rfloor$. Now, consider a generator $u$ of type (1) defined by $r_i=\lfloor k/2\rfloor$ for any $s+1\leq i\leq m$, $r_i=\lceil k/3\rceil$ for any $3\leq i\leq s$, and $r_1+r_2=k-r_3$ with $r_j\geq k-2\lfloor k/2\rfloor$ for $j\in[2]$. 
Then $u$ is a minimal generator with $\deg(u)=k+(s-3)\lceil k/3\rceil+(m-s)\lfloor k/2\rfloor+(n-m)\lceil k/2\rceil\leq (3n-s)\lceil k/6\rceil$.
Thus we obtain $\alpha(I^{(k)})\leq (3n-s)\lceil k/6\rceil$, which completes the proof.
\end{proof}


\section*{Acknowledgement}
The research of A. Ficarra was supported by the Grant JDC2023-051705-
I funded by MICIU/AEI/10.13039/501100011033 and by the FSE+ and also by
INDAM (Istituto Nazionale Di Alta Matematica).
The research of Y. Muta was partially supported by ohmoto-ikueikai and JST SPRING Japan Grant Number JPMJSP2126. 



\end{document}